\documentclass[10pt]{amsart}
\usepackage{amsfonts,amscd,amsthm,amsgen,amsmath,amssymb}
\usepackage[all]{xy}
\usepackage[vcentermath]{youngtab}
\newtheorem{theorem}{Theorem}[section]
\newtheorem{lemma}[theorem]{Lemma}
\newtheorem{corollary}[theorem]{Corollary}
\newtheorem{proposition}[theorem]{Proposition}

\theoremstyle{definition}
\newtheorem{definition}[theorem]{Definition}

\theoremstyle{remark}
\newtheorem{remark}[theorem]{Remark}

\numberwithin{equation}{section}

\begin{document}

\title[Automorphism and isometry groups of Higgs bundle moduli spaces]{Classification of the automorphism and isometry groups of Higgs bundle moduli spaces}
\author{David Baraglia}

% Address of record for the research reported here
\address{School of Mathematical Sciences, The University of Adelaide, Adelaide SA 5005, Australia}

\email{david.baraglia@adelaide.edu.au}

\begin{abstract}
Let $\mathcal{M}_{n,d}$ be the moduli space of semi-stable rank $n$, trace-free Higgs bundles with fixed determinant of degree $d$ on a Riemann surface of genus at least $3$. We determine the following automorphism groups of $\mathcal{M}_{n,d}$: (i) the group of automorphisms as a complex analytic variety, (ii) the group of holomorphic symplectomorphisms, (iii) the group of K\"ahler isomorphisms, (iv) the group of automorphisms of the quaternionic structure, (v) the group of hyper-K\"ahler isomorphisms. When $n$ and $d$ are coprime we show that $\mathcal{M}_{n,d}$ admits an anti-holomorphic isomorphism if and only if the corresponding Riemann surface admits such a map. We then use this to determine the isometry group of $\mathcal{M}_{n,d}$.
\end{abstract}
\thanks{This work is supported by the Australian Research Council Discovery Project DP110103745.}


% You can replace \today by manually entering the date (also you can manually enter to put date in format Day Month Year)
\date{\today}

% \keywords{Differential geometry, algebraic geometry}

%%%%%%%%%%%%%%%%%%%%%%%%%%%%%%%%%%%%%%%%%%%%%%%%%%%%%%%%%%%%%%%%%%%%%%%%%%%%%%%%
%%%%%%%%%%%%%%%%%%%%%%%%%%%%%%%%%%%%%%%%%%%%%%%%%%%%%%%%%%%%%%%%%%%%%%%%%%%%%%%%
%\begin{abstract}
%\end{abstract}
%%%%%%%%%%%%%%%%%%%%%%%%%%%%%%%%%%%%%%%%%%%%%%%%%%%%%%%%%%%%%%%%%%%%%%%%%%%%%%%%
%%%%%%%%%%%%%%%%%%%%%%%%%%%%%%%%%%%%%%%%%%%%%%%%%%%%%%%%%%%%%%%%%%%%%%%%%%%%%%%%

\maketitle

%%%%%%%%%%%%%%%%%%%%%%%%%%%%%%%%%%%%%%%%%%%%%%%%%%%%%%%%%%%%%%%%%%%%%%%%%%%%%%%%
%%%%%%%%%%%%%%%%%%%%%%%%%%%%%%%%%%%%%%%%%%%%%%%%%%%%%%%%%%%%%%%%%%%%%%%%%%%%%%%%
%%%%%%%%%%%%%%%%%%%%%%%%%%%%%%%%%%%%%%%%%%%%%%%%%%%%%%%%%%%%%%%%%%%%%%%%%%%%%%%%
%%%%%%%%%%%%%%%%%%%%%%%%%%%%%%%%%%%%%%%%%%%%%%%%%%%%%%%%%%%%%%%%%%%%%%%%%%%%%%%%

\section{Introduction}

Introduced by Hitchin in his groundbreaking 1987 paper \cite{hit1}, the moduli spaces of semi-stable Higgs bundles on a compact Riemann surface $\Sigma$ are complex algebraic varieties possessing a remarkable wealth of geometric structures. Most notably these moduli spaces are algebraically completely integrable systems and their smooth points are hyper-K\"ahler manifolds. As with their K\"ahler counterparts, the moduli spaces of semi-stable bundles on $\Sigma$, the study of Higgs bundle moduli spaces has proven to be an extremely fruitful pursuit. Through the non-abelian Hodge theory developed by Corlette, Donaldson, Hitchin, Simpson and others \cite{cor,don,hit1,sim2}, the Higgs bundle moduli spaces can be identified with the character varieties of representations of the fundamental group of $\Sigma$ into a complex reductive Lie group. In another direction, it has been proposed that the geometric Langlands correspondence should be understood as a mirror symmetry of Higgs bundle moduli spaces for Langlands dual groups \cite{ht,kw}. Consequently there is considerable impetus for the study of Higgs bundle moduli spaces.\\

Given a holomorphic line bundle $L_0$ on $\Sigma$, we let $\mathcal{M}_{n,L_0}$ denote the moduli space of rank $n$, trace-free Higgs bundles with determinant $L_0$. In this paper we determine the automorphism group of the moduli space $\mathcal{M}_{n,L_0}$ as a complex analytic variety, provided the genus $g$ is $\Sigma$ is at least $3$. We also determine several related symmetry groups of $\mathcal{M}_{n,L_0}$; the group of holomorphic symplectomorphisms, the group of K\"ahler isomorphisms, the group of automorphisms of the quaternionic structure and the group of hyper-K\"ahler isomorphisms. In the case that $n$ and the degree of $L_0$ are coprime, we also determine the anti-holomorphic automorphisms of $\mathcal{M}_{n,L_0}$ and the group of isometries of $\mathcal{M}_{n,L_0}$.\\

To put our results in context, we recall the classification of the automorphism group of $\mathcal{SU}_{n,L_0}$, the moduli space of rank $n$ stable bundles $E$ with fixed determinant $det(E) \cong L_0$:  

\begin{theorem}[Kouvidakis-Pantev \cite{kp}, Hwang-Ramanan \cite{hr}]\label{theoremstableauto} 
For $g \ge 3$ the automorphism group $Aut(\mathcal{SU}_{n,L_0})$ of $\mathcal{SU}_{n,L_0}$ is generated by the following automorphisms: 
\begin{enumerate} 
\item{$E \mapsto \sigma^*E \otimes L$, where $\sigma : \Sigma \to \Sigma$ is an automorphism and $L$ is a holomorphic line bundle with $L^n \cong L_0 \otimes \sigma^* L_0^*$,} 
\item{$E \mapsto \sigma^*E^* \otimes L$, where $\sigma : \Sigma \to \Sigma$ is again an automorphism and $L$ is a holomorphic line bundle with $L^n \cong L_0 \otimes \sigma^* L_0$.} 
\end{enumerate} 
Note that automorphisms of type (2) can only occur if $n$ divides $2d$, where $d = deg(L_0)$ is the degree of $L_0$. 
\end{theorem} 
Evidently the automorphism group of the moduli space $\mathcal{SU}_{n,L_0}$ is finite, as long as $g \ge 3$. In contrast we will see that the Higgs bundle moduli space $\mathcal{M}_{n,L_0}$ has a much richer automorphism group. Not only is this group infinite, it is in a sense infinite dimensional. To describe this group we consider three ways of producing automorphisms of $\mathcal{M}_{n,L_0}$: \\

{\bf 1. The $\mathbb{C}^*$-action.} There is a natural $\mathbb{C}^*$-action on $\mathcal{M}_{n,L_0}$ of the form $( E,\Phi ) \mapsto (E , \lambda \Phi)$, for $\lambda \in \mathbb{C}^*$. This action has proven to be an indispensable tool in the study of the topology of Higgs bundle moduli spaces. Notably, the circle subgroup $U(1) \subset \mathbb{C}^*$ has an associated moment map which has been used extensively to study $\mathcal{M}_{n,L_0}$ through Morse-theoretic techniques \cite{hit1,got,bgg}. \\

{\bf 2. Automorphisms of $\mathcal{SU}_{n,L_0}$.} Let $\mathcal{SU}_{n,L_0}^{\rm sm}$ denote the smooth points of $\mathcal{SU}_{n,L_0}$. The second class of automorphisms we wish to consider exploits the natural inclusion $T^* \mathcal{SU}^{\rm sm}_{n,L_0} \subset \mathcal{M}_{n,L_0}$ of the cotangent bundle of $\mathcal{SU}_{n,L_0}^{\rm sm}$ as a dense open subset of $\mathcal{M}_{n,L_0}$. By differentiation, an automorphism $\phi : \mathcal{SU}_{n,L_0} \to \mathcal{SU}_{n,L_0}$ defines an automorphism of the cotangent bundle $\phi_* = (\phi^*)^{-1} : T^* \mathcal{SU}^{\rm sm}_{n,L_0} \to \mathcal{SU}^{\rm sm}_{n,L_0}$. Clearly the automorphisms described in Theorem \ref{theoremstableauto} extend to the whole moduli space $\mathcal{M}_{n,L_0}$. In this way $Aut( \mathcal{SU}_{n,L_0} )$ can be identified with a subgroup of $Aut( \mathcal{M}_{n,L_0})$.\\

{\bf 3. Vertical flows of the Hitchin system.} To describe these automorphisms we first recall the {\em Hitchin system} \cite{hit2}. The Higgs bundle moduli space carries a naturally defined holomorphic symplectic form $\Omega_I$ extending the canonical symplectic structure on $T^* \mathcal{SU}^{\rm sm}_{n,L_0}$. As we shall recall in the paper, $\mathcal{M}_{n,L_0}$ is an algebraically completely integrable system. Specifically, there exists holomorphic Poisson commuting functions $h_1 , \dots , h_m $ on $\mathcal{M}_{n,L_0}$, where $m = \tfrac{1}{2} dim( \mathcal{M}_{n,L_0} )$, such that $dh_1 , \dots , dh_m$ are generically independent and the generic fibre of the Hitchin map 
\begin{equation*}
h = (h_1 , \dots , h_m ) : \mathcal{M}_{n,L_0} \to \mathbb{C}^m
\end{equation*}
is an abelian variety. Let $\mathcal{A} = \mathbb{C}^m$ denote the base of the Hitchin system. To each holomorphic function $f$ on $\mathcal{A}$ there is an associated Hamiltonian vector field $X_f$. The properness of the Hitchin map ensures that $X_f$ can be integrated to a symplectomorphism $e^{X_f}$ of $\mathcal{M}^{\rm sm}_{n,L_0}$, the smooth locus of $\mathcal{M}_{n,d}$. We will see that $e^{X_\mu}$ actually extends to an automorphism of the full moduli space. More generally if $\mu$ is a holomorphic $1$-form on $\mathcal{A}$ we have an associated vector field $X_\mu$ determined by the relation $i_{X_\mu} \Omega_I = h^*(\mu)$. Such a vector field can likewise be integrated to an automorphism $e^{X_\mu}$ of $\mathcal{M}^{\rm sm}_{n,L_0}$, which is not a symplectomorphism unless $\mu$ is exact. Similarly we find that $e^{X_\mu}$ extends to an automorphism of the full moduli space. We also show that the set of all automorphisms of the form $e^{X_\mu}$ gives an abelian subgroup of $Aut(\mathcal{M}_{n,L_0})$, isomorphic to the set of holomorphic $1$-forms on $\mathcal{A}$ under addition. We denote this subgroup by $Vert_0(\mathcal{M}_{n,L_0}) \subset Aut(\mathcal{M}_{n,L_0})$. While Hamiltonian flows of the Hitchin system have been considered by several authors (e.g., \cite{hit2,hk,bzn}), the larger group $Vert_0(\mathcal{M}_{n,L_0})$ seems to have been given little consideration. The group $Vert_0(\mathcal{M}_{n,L_0})$ bears a close relationship to a construction central to Ng\^o's proof of the fundamental lemma. Notice that on the non-singular fibres of $h : \mathcal{M}_{n,L_0} \to \mathcal{A}$, which are abelian varieties, the group $Vert_0(\mathcal{M}_{n,L_0})$ acts by translation on each fibre. This is similar to the construction in \cite{ngo1,ngo2} of a group scheme $P^{\rm ell} \to \mathcal{A}^{\rm ell}$, over the so-called elliptic locus $\mathcal{A}^{\rm ell} \subset \mathcal{A}$ which carries a fibrewise action $P^{\rm ell} \times_{\mathcal{A}^{\rm ell}} \mathcal{M}_{n,L_0}^{\rm ell} \to \mathcal{M}_{n,L_0}^{\rm ell}$. On the non-singular fibres this action is again by translations.\\

Our main theorem is that the above three classes of automorphisms generate the automorphism group. More precisely we have:

\begin{theorem}\label{theoremmain0}
Let $\Sigma$ have genus $g \ge 3$. The subgroup $Vert_0(\mathcal{M}_{n,L_0}) \subset Aut(\mathcal{M}_{n,L_0})$ is normal. Moreover we have an isomorphism:
\begin{equation*}
Aut(\mathcal{M}_{n,L_0}) = \left( \mathbb{C}^* \times Aut(\mathcal{SU}_{n,L_0}) \right) \ltimes Vert_0( \mathcal{M}_{n,L_0} ).
\end{equation*}
\end{theorem}

Using an infinite dimensional hyper-K\"ahler quotient construction \cite{hit1}, Hitchin showed that on the smooth points $\mathcal{M}_{n,L_0}^{\rm sm}$ of $\mathcal{M}_{n,L_0}$, there is a natural hyper-K\"ahler structure. To be specific, this consists of integrable complex structures $I,J,K$ satisfying the quaternionic relation $IJ = K$ and a Riemannian metric $g$ which is K\"ahler with respect to $I,J$ and $K$. In terms of the associated K\"ahler forms $\omega_I,\omega_J,\omega_K$, the holomorphic symplectic form $\Omega_I$ is given by $\Omega_I = \omega_J + i\omega_K$. In addition to the automorphisms of $\mathcal{M}_{n,L_0}$ as a complex analytic variety, we also consider various subgroups of $Aut(\mathcal{M}_{n,L_0})$ preserving different parts of the hyper-K\"ahler structure:

\begin{definition}\label{defautgroups0}
We define the following subgroups of $Aut(\mathcal{M}_{n,L_0})$:
\begin{itemize}
\item{$Aut_{Sympl}(\mathcal{M}_{n,L_0}) = \{ \phi \in Aut(\mathcal{M}_{n,L_0}) \, | \, \phi|_{\mathcal{M}_{n,L_0}^{\rm sm}} \text{ preserves } \Omega_I \}$, the group of holomorphic symplectomorphisms of $\mathcal{M}_{n,L_0}$.}
\item{$Aut_{Isom}(\mathcal{M}_{n,L_0}) = \{ \phi \in Aut(\mathcal{M}_{n,L_0}) \, | \, \phi|_{\mathcal{M}_{n,L_0}^{\rm sm}} \text{ preserves } g \}$, the group of holomorphic isometries of $\mathcal{M}_{n,L_0}$.}
\item{$Aut_{Q}(\mathcal{M}_{n,L_0}) = \{ \phi \in Aut(\mathcal{M}_{n,L_0}) \, | \, \phi|_{\mathcal{M}_{n,L_0}^{\rm sm}} \text{ preserves } J \}$, the group of quaternionic isomorphisms of $\mathcal{M}_{n,L_0}$.}
\item{$Aut_{HK}(\mathcal{M}_{n,L_0}) = \{ \phi \in Aut(\mathcal{M}_{n,L_0}) \, | \, \phi|_{\mathcal{M}_{n,L_0}^{\rm sm}} \text{ preserves } g,J \}$, the group of hyper-K\"ahler isomorphisms of $\mathcal{M}_{n,L_0}$.}
\end{itemize}
\end{definition}

Our second main theorem is a complete description of these subgroups:

\begin{theorem}\label{theoremgroups}
Under the isomorphism $Aut(\mathcal{M}_{n,L_0}) \cong \left( \mathbb{C}^* \times Aut(\mathcal{SU}_{n,L_0}) \right) \ltimes Vert_0(\mathcal{M}_{n,L_0})$, the subgroups given in Definition \ref{defautgroups0} are as follows:
\begin{enumerate}
\item{$Aut_{Sympl}(\mathcal{M}_{n,L_0}) = \left( \{1\} \times Aut( \mathcal{SU}_{n,L_0} ) \right) \ltimes Ham(\mathcal{M}_{n,d})$,}
\item{$Aut_{Isom}(\mathcal{M}_{n,L_0}) = \left( U(1) \times Aut( \mathcal{SU}_{n,L_0} ) \right)$,}
\item{$Aut_{Q}(\mathcal{M}_{n,L_0}) = \left( \mathbb{R}_{+} \times Aut( \mathcal{SU}_{n,L_0} ) \right)$,}
\item{$Aut_{HK}(\mathcal{M}_{n,L_0}) = \left( \{ 1 \} \times Aut( \mathcal{SU}_{n,L_0} ) \right)$,}
\end{enumerate}
where $U(1) \subset \mathbb{C}^*$ is the subgroup of unit complex numbers and $\mathbb{R}_+ \subset \mathbb{C}^*$ is the subgroup of positive real numbers.
\end{theorem}

In this theorem, $Ham(\mathcal{M}_{n,L_0})$ is the subgroup of $Vert_0(\mathcal{M}_{n,L_0})$ consisting of the Hamiltonian flows associated to holomorphic functions on $\mathcal{A}$.\\

We consider two further kinds of symmetries associated to $\mathcal{M}_{n,L_0}$. For this we will assume that $d = deg(L_0)$ is coprime to $n$, so that $\mathcal{M}_{n,L_0}$ is a smooth hyper-K\"ahler manifold. By an {\em anti-automorphism} of a complex manifold $(X,I)$, we mean a diffeomorphism $\phi : X \to X$ such that $\phi_* \circ I = -I \circ \phi_*$.

\begin{theorem}\label{theoremanti0}
Suppose that $n$ and $d$ are coprime. Then $\mathcal{M}_{n,d}$ admits an anti-automorphism if and only if $\Sigma$ admits an anti-automorphism.
\end{theorem}

Using this, we are able to determine the full isometry group of $\mathcal{M}_{n,L_0}$:

\begin{theorem}\label{theoremisometry0}
Suppose $n$ and $d = deg(L_0)$ are coprime.
\begin{itemize}
\item{If $\Sigma$ does not admit an anti-automorphism, then every isometry of $\mathcal{M}_{n,L_0}$ preserves $I$. Therefore $Isom(\mathcal{M}_{n,L_0}) = Aut_{Isom}(\mathcal{M}_{n,L_0}) \cong \left( U(1) \times Aut( \mathcal{SU}_{n,L_0} ) \right)$.}
\item{If $\Sigma$ admits an anti-automorphism then the subgroup of isometries of $\mathcal{M}_{n,d}$ preserving $I$ has index $2$ in the isometry group of $\mathcal{M}_{n,L_0}$.}
\end{itemize}
\end{theorem}

In the case that $\Sigma$ admits an anti-automorphism we can say more precisely what the isometry group $Isom(\mathcal{M}_{n,L_0})$ is. We leave the details to Section \ref{secisometry2}.\\

The plan of the paper is as follows. Section \ref{sechiggsbundles} is a review of Higgs bundle moduli spaces, the Hitchin system and spectral curves. In Section \ref{sechamflows} we study in depth the Hamiltonian flows of the Hitchin system on non-singular fibres in \textsection \ref{sechamflow1} and then on generic singular fibres in \textsection \ref{sechamflow2}. This will allow us to deduce that the Hitchin map is a submersion on a suitably large open subset. In Section \ref{secholovf}, we give a classification of the holomorphic vector fields. To do this, we consider in \textsection \ref{secksm} the Kodaira-Spencer mapping associated to the non-singular fibres of the Hitchin system. This allows us in \textsection \ref{secchvf} to determine the holomorphic vector fields on the complement of the singular fibres. Using the results of Section \ref{sechamflows} we can then deduce which holomorphic vectors fields extend to the entire smooth locus of the moduli space. Based on our study of flows of the Hitchin system, we are further able to determine which holomorphic vector fields integrate to automorphisms of the smooth locus and by inspection we see that these extend as automorphisms to the whole moduli space.\\

In Section \ref{secpmt}, we prove our main result, Theorem \ref{theoremmain0}. The key idea is that if $\phi$ is an automorphism of $\mathcal{M}_{n,L_0}$, then we can find a flow $e^{X_\nu}$ along the fibres of the Hitchin system such that the composition $e^{X_\nu} \circ \phi$ commutes with the $\mathbb{C}^*$-action. We are able to accomplish this by using our classification of the holomorphic vector fields on $\mathcal{M}^{\rm sm}_{n,L_0}$. This reduces the problem to the classification of automorphisms commuting with the $\mathbb{C}^*$-action. We are then able to reduce this to the problem of determining automorphisms of the cotangent bundle $T^* \mathcal{SU}_{n,d}^{\rm sm}$ which are linear in the fibres. Using once again our classification of holomorphic vector fields on the moduli space, we are able to determine the group of all such automorphisms and the main theorem follows. In Section \ref{secstructures} we give the proof of Theorem \ref{theoremgroups}. Our main strategy is to examine how automorphisms which preserve some of the geometry of $\mathcal{M}_{n,L_0}$ act on the rest of the geometry. Our classification of holomorphic vector fields on $\mathcal{M}_{n,L_0}^{\rm sm}$ turns out to be crucial to this approach. Finally in Section \ref{secisometry} we consider anti-automorphisms and isometries of the moduli space under the assumption of coprime rank and degree. In \textsection \ref{secanti} we prove Theorem \ref{theoremanti0} as a consequence of the Torelli theorem for Higgs bundle moduli spaces \cite{bigo}. In \textsection \ref{secisometry2} we prove Theorem \ref{theoremisometry0} and consequently a classification of the isometry groups of these moduli spaces.

%%%%%%%%%%%%%%%%%%%%%%%%%%%%%%%%%%%%%%%%%%%%%%%%%%%%%%%%%%%%%%%%%%%%%%%%%%%%%%%%
%%%%%%%%%%%%%%%%%%%%%%%%%%%%%%%%%%%%%%%%%%%%%%%%%%%%%%%%%%%%%%%%%%%%%%%%%%%%%%%%

\section{Higgs bundles and the Hitchin system}\label{sechiggsbundles}

%%%%%%%%%%%%%%%%%%%%%%%%%%%%%%%%%%%%%%%%%%%%%%

\subsection{Higgs bundle moduli spaces}\label{sechbms}

Let $\Sigma$ be a compact Riemann surface of genus $g > 1$ and let $K$ denote the canonical bundle of $\Sigma$. A {\em Higgs bundle} of rank $n$, degree $d$ on $\Sigma$ is a pair $(E,\Phi$), where $E$ is a holomorphic vector bundle of rank $n$, degree $d$ and $\Phi$ is a holomorphic section of $End(E) \otimes K$, called the {\em Higgs field}. It is often convenient to think of $E$ as consisting of an underlying $\mathcal{C}^\infty$ vector bundle, which we also denote by $E$, together with a choice of holomorphic structure on $E$. The holomorphic structure on $E$ is specified by giving a $\overline{\partial}$-operator, $\overline{\partial}_E : \Omega^0(\Sigma,E) \to \Omega^{0,1}(\Sigma,E)$ and the requirement that $\Phi$ is holomorphic reads $\overline{\partial}_E \Phi = 0$. As such we will often denote Higgs bundles as pairs $(\overline{\partial}_E , \Phi)$.\\

The {\em slope} $\mu(E)$ of a holomorphic vector bundle $E$ on $\Sigma$ is defined as the quotient $\mu(E) = deg(E)/rank(E)$. Similarly the slope $\mu(E,\Phi)$ of a Higgs bundle $(E,\Phi)$ is simply the slope of $E$, $\mu(E,\Phi) = \mu(E)$. Recall that a Higgs bundle $(E,\Phi)$ is said to be {\em semi-stable} if for all proper $\Phi$-invariant subbundles $F \subset E$, we have $\mu(F) \le \mu(E)$. We say that $(E,\Phi)$ is {\em stable} if this inequality is strict for all such $F$. Any semi-stable Higgs bundle $(E,\Phi)$ has a Jordan-H\"older filtration, that is, a sequence $0 = E_1 \subset \dots \subset E_l = E$ of $\Phi$-invariant subbundles for which $\mu(E_i/E_{i-1}) = \mu(E)$ and the induced Higgs bundles $(E_i/E_{i-1} , \Phi_i)$ are stable. The associated graded Higgs bundle $gr(E,\Phi) = \bigoplus_i (E_i/E_{i-1} , \Phi_i)$ is determined up to isomorphism from $(E,\Phi)$ \cite{nit}. Two semi-stable Higgs bundles are said to be {\em S-equivalent} if their associated graded Higgs bundles are isomorphic. In particular, two stable Higgs bundles are S-equivalent if and only if they are isomorphic.\\

We say that a Higgs bundle $(E,\Phi)$ is trace-free if $\Phi$, viewed as a $K$-valued endomorphism of $E$ is trace-free. Suppose that $L_0$ is a holomorphic line bundle of degree $d$. We say that $(E,\Phi)$ has determinant $L_0$ if $det(E) \cong L_0$.

%%%%%%%%%%%%%%%%%%%%%%%%%%%%%%%%%%%%%%
% DEFINITION
%%%%%%%%%%%%%%%%%%%%%
\begin{definition}
We let $\mathcal{M}_{n,L_0}$ denote the moduli space of $S$-equivalence classes of rank $n$, degree $d$ Higgs bundles, which are trace-free and have determinant $L_0$. The space $\mathcal{M}_{n,L_0}$ is a quasi-projective algebraic variety over $\mathbb{C}$ and has an open subvariety $\mathcal{M}_{n,L_0}^{s}$, the moduli space of stable rank $n$, degree $d$ Higgs bundle, trace-free with determinant $L_0$ \cite{nit}. The dimension of $\mathcal{M}_{n,d}$ is $2(n^2-1)(g-1)$.
\end{definition}
%%%%%%%%%%%%%%%%%%%%%%%%%%

\begin{remark}
While $\mathcal{M}_{n,L_0}$ has an algebraic structure, we will study the automorphisms of $\mathcal{M}_{n,L_0}$ within the category of complex analytic varieties. This is an important point to bear this in mind, as $\mathcal{M}_{n,L_0}$ is non-compact.
\end{remark}

\begin{remark}
For any line bundle $L$, the tensor product $(E,\Phi) \mapsto (E \otimes L , \Phi \otimes Id )$ defines an isomorphism $\otimes L : \mathcal{M}_{n,L_0} \to \mathcal{M}_{n,L_0 \otimes L^n}$. It follows that if $L_0,L'_0$ have the same degree modulo $n$, there is an isomorphism $\mathcal{M}_{n,L_0} \cong \mathcal{M}_{n,L'_0}$. As an algebraic variety, $\mathcal{M}_{n,L_0}$ only depends on $L_0$ through the degree $d = deg(L_0)$ and as such we will often use $\mathcal{M}_{n,d}$ to denote $\mathcal{M}_{n,L_0}$, where $L_0$ is any holomorphic line bundle of degree $d$.
\end{remark}

Let $\mathcal{M}_{n,d}^{\rm sm} \subseteq \mathcal{M}_{n,d}$ denote the locus of smooth points of $\mathcal{M}_{n,d}$ and note that stable points are smooth, so $\mathcal{M}_{n,d}^s \subseteq \mathcal{M}_{n,d}^{\rm sm}$. We also note that when $n$ and $d$ are coprime, every semistable Higgs bundle is stable so that $\mathcal{M}_{n,d}^s = \mathcal{M}_{n,d}^{\rm sm} = \mathcal{M}_{n,d}$ and the moduli space is a complex manifold. For a smooth point $m \in \mathcal{M}_{n,d}^{\rm sm}$ represented by a Higgs bundle pair $(E,\Phi)$, the tangent space $T_m \mathcal{M}_{n,d}^{\rm sm}$ may be described as follows. Let $End_0(E) \subset End(E)$ denotes the subbundle of trace-free endomorphisms of $E$. An infinitesimal deformation of $(E,\Phi)$ is represented by a pair $(\dot{A} , \dot{\Phi} ) \in \Omega^{0,1}(\Sigma , End_0(E)) \oplus \Omega^{1,0}(\Sigma , End_0(E))$, where $\dot{A}$ is a deformation of the holomorphic structure $\overline{\partial}_E$ and $\dot{\Phi}$ is a deformation of $\Phi$. Differentiating the condition $\overline{\partial}_E \Phi = 0$, we see that such pairs $(\dot{A} , \dot{\Phi})$ must satisfy $\overline{\partial}_E \dot{\Phi} + [\dot{A} , \Phi] = 0$. Further, a pair $(\dot{A} , \dot{\Phi})$ gives a trivial deformation whenever $(\dot{A} , \dot{\Phi}) = (\overline{\partial}_E \psi , [\Phi , \psi])$, for some $\psi \in \Omega^0(\Sigma , End_0(E))$. It can then be shown that this gives a natural identification of $T_m \mathcal{M}_{n,d}^{\rm sm}$ with the space of pairs $(\dot{A} , \dot{\Phi})$ such that $\overline{\partial}_E \dot{\Phi} + [\dot{A} , \Phi] = 0$, modulo pairs of the form $(\dot{A} , \dot{\Phi}) = (\overline{\partial}_E \psi , [\Phi , \psi])$. This is precisely the degree $1$ hypercohomology of the complex
\begin{equation*}
\xymatrix{
End_0(E) \ar[r]^-{ [\Phi , \, . \, \, ] } & End_0(E) \otimes K.
}
\end{equation*}

The smooth locus $\mathcal{M}_{n,d}^{\rm sm}$ is a complex manifold. We let $I$ denote the complex structure on $\mathcal{M}_{n,d}^{\rm sm}$. In terms of deformations $(\dot{A}, \dot{\Phi})$, the complex structure is simply the natural complex structure $I(\dot{A} , \dot{\Phi}) = (i\dot{A} , i\dot{\Phi})$. Furthermore $\mathcal{M}_{n,d}^{\rm sm}$ carries a natural holomorphic symplectic form $\Omega_I$, which may be defined as follows:
\begin{equation}\label{equomegai1}
\Omega_I ( (\dot{A}_1 , \dot{\Phi}_1) , (\dot{A}_2 , \dot{\Phi}_2 ) ) = \int_\Sigma Tr( \dot{A}_1 \dot{\Phi}_2 - \dot{A}_2 \dot{\Phi}_1 ),
\end{equation}
where $Tr$ denotes the trace of an endomorphism. The $2$-form $\Omega_I$ is closed and defines a holomorphic symplectic structure on $\mathcal{M}_{n,d}^{\rm sm}$. One way to see this is to construct $\Omega_I$ by means of an infinite dimensional symplectic quotient \cite{hit1,hit2} or indeed a hyper-K\"ahler quotient. We will consider the hyper-K\"ahler structure of $\mathcal{M}_{n,d}^{\rm sm}$ in Section \ref{secstructures}. For now we only need to make use of the holomorphic symplectic structure $(I,\Omega_I)$.\\

A Higgs bundle $(E,0)$ with zero Higgs field is simply a holomorphic vector bundle $E$. For such Higgs bundles (semi)-stability reduces to the usual definition of (semi)-stability of holomorphic bundles. There is likewise a moduli space $\mathcal{SU}_{n,L_0}$ of S-equivalence classes of rank $n$, degree $d$ semi-stable bundles with determinant $L_0$. As with Higgs bundles, this space only depends on $L_0$ through the degree mod $n$, so we will also denote the moduli space by $\mathcal{SU}_{n,d}$. This space is a complex projective variety of dimension $(n^2-1)(g-1)$. We let $\mathcal{SU}_{n,d}^{\rm sm} \subseteq \mathcal{SU}_{n,d}$ denote the smooth locus. Then for a point $[E] \in \mathcal{SU}_{n,d}^{\rm sm}$ the tangent space is $T_{[E]}\mathcal{SU}_{n,d}^{\rm sm} \cong H^1( \Sigma , End_0(E) )$. By Serre duality, the cotangent space is $T^*_{[E]} \mathcal{SU}_{n,d}^{\rm sm} \cong H^0( \Sigma , End_0(E) \otimes K)$. Thus a cotangent vector $\Phi \in T^*_{[E]} \mathcal{SU}_{n,d}^{\rm sm}$ can be thought of as a Higgs field on $E$ and one sees that $(E,\Phi)$ defines a point in $\mathcal{M}_{n,d}^{{\rm sm}}$. In this way we obtain a natural open inclusion $T^* \mathcal{SU}_{n,d}^{\rm sm} \subset \mathcal{M}_{n,d}^{\rm sm}$. Note also that the symplectic form $\Omega_I$ on $\mathcal{M}_{n,d}^{\rm sm}$ restricts to the canonical symplectic form on $T^*\mathcal{SU}_{n,d}^{\rm sm}$.

%%%%%%%%%%%%%%%%%%%%%%%%%%%%%%%%%%%%%%%%%%%%%%%%%%%%%%%%%%%%%%%%%%%%%%%%%
%%%%%%%%%%%%%%%%%%%%%%%%%%%%%%%%%%%%%%%%%%%%%%%%%%%%%%%%%%%%%%%%%%%%%%%%%

\subsection{Spectral curves and the Hitchin system}\label{secschs}

Identify the Lie algebra $\mathfrak{sl}(n,\mathbb{C})$ with the space of trace-free endomorphisms of $\mathbb{C}^n$. Any $A \in \mathfrak{sl}(n,\mathbb{C})$ has a characteristic polynomial:
\begin{equation*}
det(\lambda - A ) = \lambda^n + a_2(A) \lambda^{n-2} + \dots + a_n(A).
\end{equation*}
The coefficients $a_2, \dots , a_n$ are generators for the ring of polynomials on $\mathfrak{sl}(n,\mathbb{C})$ which are invariant under the adjoint action of $SL(n,\mathbb{C})$. Note that $a_j$ is homogeneous of degree $j$ and we can alternatively define $a_j$ as:
\begin{equation}\label{equapoly}
a_j(A) = (-1)^j Tr_{ \wedge^j \mathbb{C}^n } ( A ),
\end{equation}
where for any $\mathfrak{sl}(n,\mathbb{C})$-representation $R$, we let $Tr_R(A)$ denote the trace of $A$ in the representation $R$. When $R = \mathbb{C}^n$ is the standard representation we write this simply as $Tr$. There are many other generating sets for the ring of invariant polynomials, we mention just two other sets $\{ b_j \}_{j=2}^n$, $\{ h_j \}_{j=2}^n$ which will be important to us:
\begin{eqnarray}
b_j(A) &=& \frac{1}{j} Tr( A^j ), \label{equbpoly} \\
h_j(A) &=& Tr_{S^j(\mathbb{C}^n)}(A). \label{equhpoly}
\end{eqnarray}
Note that $b_j$ and $h_j$ are also homogeneous of degree $j$. The invariants $a_j,b_j,h_j$ are defined for every integer $j \ge 1$ by Equations (\ref{equapoly})-(\ref{equhpoly}). Since $A$ is trace-free we have $a_1 = b_1 = h_1 = 0$. For convenience we also set $a_0 = h_0 = 1$ but leave $b_0$ undefined. The generating sets $\{a_j \}_{j=2}^n , \{ b_j \}_{j=2}^n , \{ h_j \}_{j=2}^n $ are related by the following versions of Newton's identities \cite{mac}, valid for any $k \ge 1$:
\begin{equation}\label{equnewton}
\begin{aligned}
\sum_{j=1}^k j b_j a_{k-j} = -ka_k, && \sum_{j=1}^k j b_j h_{k-j} = kh_k, && \sum_{j=0}^k h_j a_{k-j} = 0.
\end{aligned}
\end{equation}

Suppose that $(E,\Phi)$ is a rank $n$ trace-free Higgs bundle. If $f_j$ is any degree $j$ invariant polynomial on $\mathfrak{sl}(n,\mathbb{C})$, then we can apply $f_j$ to $\Phi$ to obtain a holomorphic section $f_j( \Phi )$ of $K^j$. Invariance of $f_j$ further ensures that $f_j(\Phi)$ depends only on the isomorphism class of $(E,\Phi)$. Let $f_2 , \dots , f_n$ be a set of generators for the ring of invariant polynomials, where $f_j$ is homogeneous of $j$. This defines a holomorphic map $h : \mathcal{M}_{n,d} \to \mathcal{A}$ into the affine space 
\begin{equation*}
\mathcal{A} = \bigoplus_{j=2}^n H^0(\Sigma , K^j),
\end{equation*}
by evaluating the polynomials $f_j$ on the Higgs field:
\begin{equation*}
h(E,\Phi) = ( f_2(\Phi) , f_3(\Phi) , \dots , f_n(\Phi) ).
\end{equation*}
The map $h$ is known as the {\em Hitchin map} and $h : \mathcal{M}_{n,d} \to \mathcal{A}$ is referred to as the {\em Hitchin fibration} or {\em Hitchin system}. Given two different choices of generators for the ring of invariant polynomials $\{ f_j \}, \{ f'_j \}$, the corresponding Hitchin maps $h,h'$ are related by an isomorphism of $\mathcal{A}$. In this sense the Hitchin map is essentially independent of the choice of generators. However, it will be convenient to take the Hitchin map with respect to $b_2, \dots , b_n$ in (\ref{equbpoly}). Therefore we define the Hitchin map $h : \mathcal{M}_{n,d} \to \mathcal{A}$ from now on to be given by
\begin{equation*}
h(E,\Phi) = \left( b_2(\Phi) , b_3(\Phi) , \dots , b_n(\Phi) \right) = \left( \tfrac{1}{2} Tr( \Phi^2) , \tfrac{1}{3} Tr( \Phi^3) , \dots , \tfrac{1}{n} Tr( \Phi^n ) \right).
\end{equation*}

When no confusion is likely to arise we will denote $a_j(\Phi) , b_j(\Phi) , h_j(\Phi) \in H^0(\Sigma , K^j)$ simply as $a_j,b_j, h_j$. Newton's identities (\ref{equnewton}) allow us to express the $a_j(\Phi)$ and $h_j(\Phi)$ in terms of the $b_j(\Phi)$. In particular, there is an automorphism $\alpha : \mathcal{A} \to \mathcal{A}$ of the affine variety $\mathcal{A}$ such that for any $(E,\Phi)$ we have $\alpha( b_2(\Phi) , \dots , b_n(\Phi) ) = ( a_2(\Phi) , \dots , a_n( \Phi ) )$. For a Higgs bundle $(E,\Phi) \in \mathcal{M}_{n,d}$ we will often use to $b = (b_2 , \dots , b_n) \in \mathcal{A}$ to denote $h(E,\Phi)$ and $a = (a_2 , \dots , a_n ) \in \mathcal{A}$ will denote $\alpha(b)$.\\

The Hitchin fibration is an {\em algebraically completely integrable system} \cite{hit2}. Recall that this means $h : \mathcal{M}_{n,d} \to \mathcal{A}$ is a holomorhpic Lagrangian fibration with respect to $\Omega_I$, that the generic fibre of $h$ is an open set in an abelian variety and that the Hamiltonian vector fields $X_{f_1} , X_{f_2} , \dots , X_{f_m}$ are linear on the fibres, where $f_1, \dots , f_m$ are coordinate functions on $\mathcal{A}$. As we will recall, in the case of the Hitchin fibration, the generic fibres are actually abelian varieties. This guarantees that the vector fields $X_{f_1} , \dots , X_{f_m}$ are linear on the fibres, since every global holomorphic vector field on a complex torus is linear.\\

The fibres of the Hitchin fibration can be described using the notion of spectral curves. For this, suppose we are given $b = (b_2 , \dots , b_n ) \in \mathcal{A}$ and set $a = \alpha(b) = (a_2 , \dots , a_n) \in \mathcal{A}$, so $a_j \in H^0(\Sigma , K^j )$. Let $\pi : K \to \Sigma$ denote the projection from the total space of $K$ to $\Sigma$ and let $\lambda$ denote the tautological section of $\pi^*(K)$. Consider the section $s_b \in H^0(K , \pi^*(K^n))$ given by
\begin{equation}\label{equsectiona}
s_b = \lambda^n + \pi^*(a_2) \lambda^{n-2} + \dots + \pi^*(a_n).
\end{equation}
The zero locus $S_b \subset K$ of $s_b$ is called the {\em spectral curve} associated to $b$. In general $S_b$ can be singular, however, Bertini's theorem implies that $S_b$ is smooth for generic $b \in \mathcal{A}$. Let us define the {\em discriminant divisor} $\mathcal{D} \subset \mathcal{A}$ as
\begin{equation*}
\mathcal{D} = \{ b \in \mathcal{A} \, | \, S_b \text{ is not smooth } \}.
\end{equation*}
It can be shown that $\mathcal{D}$ is an irreducible divisor in $\mathcal{A}$ \cite[Corollary 1.5]{kp}. Any Higgs bundle $(E,\Phi) \in \mathcal{M}_{n,d}$ defines a spectral curve given by the characteristic equation of $\Phi$:
\begin{equation*}
det( \lambda - \pi^*(\Phi) ) = 0.
\end{equation*}
Note that $det( \lambda - \pi^*(\Phi) ) = s_b$, where $b = h(E,\Phi)$ and $s_b$ is given by Equation (\ref{equsectiona}). So the spectral curve $det(\lambda - \pi^*(\Phi) ) = 0$ associated to $(E,\Phi)$ is precisely the spectral curve associated to $b = h(E,\Phi) \in \mathcal{A}$.\\

Let $\mathcal{A}^{\rm reg} = \mathcal{A} - \mathcal{D}$ denote the complement of the discriminant divisor. We say that $b \in \mathcal{A}$ is {\em regular} if $b \in \mathcal{A}^{\rm reg}$. Thus $b$ is regular precisely if the spectral curve $S_b$ is smooth. Similarly we let $\mathcal{M}_{n,d}^{\rm reg} = h^{-1}( \mathcal{A}^{\rm reg} )$ be the space of Higgs bundles with smooth spectral curve. Then $\mathcal{M}_{n,d}^{\rm reg} \subset \mathcal{M}_{n,d}$ is a dense open subset. Moreover one can show that every element of $\mathcal{M}_{n,d}^{\rm reg}$ is stable, so that $\mathcal{M}_{n,d}^{\rm reg} \subseteq \mathcal{M}_{n,d}^{\rm sm}$ \cite{bnr}.\\

Let $b \in \mathcal{A}^{\rm reg}$ and let $\pi : S_b \to \Sigma$ be the associated spectral curve. To simplify notation we will denote $S_b$ simply as $S$ when no confusion is likely to occur. Let $K_S$ denote the canonical bundle of $S$. By the adjunction formula we have $K_S \cong \pi^*(K^n)$. Set $\tilde{d} = d + n(n-1)(g-1)$ and let $Jac_{\tilde{d}}(S)$ denote the space of degree $\tilde{d}$ line bundles on $S$. If $L \in Jac_{\tilde{d}}(S)$ we have by Grothendieck-Riemann-Roch that $E = \pi_*(L)$ is a rank $n$, degree $d$ holomorphic vector bundle on $\Sigma$. The tautological section $\lambda$ may be viewed as a map $\lambda : L \to L \otimes \pi^* K$, which then pushes down to a map $\Phi : E \to E \otimes K$. In this way we have constructed from $L$ a Higgs bundle pair $(E,\Phi)$. One can then show that $det( \lambda - \pi^*(\Phi) ) = \lambda^n + a_2 \lambda^{n-2} + \dots + a_n$ \cite{bnr}, so $S$ is the spectral curve associated to $(E,\Phi)$. Define the Prym variety $Prym_{\tilde{d}}(S,\Sigma)$ as follows:
\begin{equation*}
Prym_{\tilde{d}}(S,\Sigma) = \{ L \in Jac_{\tilde{d}}(S) \, | \, det( \pi_* L ) \cong L_0 \, \}.
\end{equation*}
Therefore if $L \in Prym_{\tilde{d}}(S,\Sigma)$, the associated Higgs bundle $(E,\Phi)$ is trace-free with determinant $L_0$. Conversely, any $(E,\Phi) \in \mathcal{M}_{n,d}$ with associated spectral curve $S$ is obtained in this way from some $L \in Prym_{\tilde{d}}(S,\Sigma)$ \cite{hit2,bnr}. This shows that the fibre $h^{-1}(b)$ of the Hitchin system over $b$ is the abelian variety $Prym_{\tilde{d}}(S,\Sigma)$. Let $Nm : Jac(S) \to Jac(\Sigma)$ denote the norm map associated to $\pi : S \to \Sigma$. Then for any $L \in Jac(S)$ we have $Nm(L) = det( \pi_* L ) \otimes K^{n(n-1)/2}$ \cite{bnr}, hence we can alternatively characterise $Prym_{\tilde{d}}(S,\Sigma)$ as those line bundles $L \in Jac_{\tilde{d}}(S)$ with $Nm(L) = L_0 \otimes K^{n(n-1)/2}$.

%%%%%%%%%%%%%%%%%%%%%%%%%%%%%%%%%%%%%%%%%%%%%%%%%%%%%%%%%%%%%%%%%%%%%
%%%%%%%%%%%%%%%%%%%%%%%%%%%%%%%%%%%%%%%%%%%%%%%%%%%%%%%%%%%%%%%%%%%%%

\section{Hamiltonian flows of the Hitchin system}\label{sechamflows}

%%%%%%%%%%%%%%%%%%%%%%%%%%%%%%%%%%%%%%%%%%

\subsection{Flows on non-singular fibres}\label{sechamflow1}

Given a holomorphic function $f : \mathcal{A} \to \mathbb{C}$ on $\mathcal{A}$, we let $X_f \in H^0( \mathcal{M}_{n,d}^{\rm sm} , T\mathcal{M}_{n,d}^{{\rm sm}})$ denote the corresponding Hamiltonian vector field on $\mathcal{M}_{n,d}^{{\rm sm}}$, given by the relation:
\begin{equation*}
i_{X_f} \Omega_I = h^*(df).
\end{equation*}
It will be useful to have a more explicit description of these Hamiltonian vector fields. Consider a point $b \in \mathcal{A}^{{\rm reg}}$ and let $\pi : S \to \Sigma$ be the corresponding spectral curve. The fibre $h^{-1}(b)$ of the Hitchin system over $b$ is $Prym_{\tilde{d}}(S,\Sigma) \subset Jac_{\tilde{d}}(S)$, so for each $m \in h^{-1}(b)$ the vertical tangent space $T_m h^{-1}(b)$ can be canonically identified with the kernel of $Nm_* : H^1( S , \mathcal{O}_S ) \to H^1( \Sigma , \mathcal{O}_\Sigma )$. Under this identification $X_f(m)$ is an element of $H^1(S , \mathcal{O}_S)$.\\

Recall that $\mathcal{A} = \bigoplus_{j=2}^n H^0(\Sigma , K^j)$ and define $\mathcal{A}^* = \bigoplus_{j=2}^n H^1(\Sigma , K^{1-j})$. Serre duality applied termwise defines a dual pairing $\langle \, \, , \, \, \rangle : \mathcal{A}^* \otimes \mathcal{A} \to \mathbb{C}$.

%%%%%%%%%%%%%%
%       PROPOSITION      %
%%%%%%%%%%%%%%
\begin{proposition}
Given $b = (b_2 , \dots , b_n) \in \mathcal{A}^{{\rm reg}}$, let $\gamma_b : \mathcal{A}^* = \bigoplus_{j=2}^n H^1(\Sigma, K^{1-j})  \to H^1(S, \mathcal{O}_S)$ be given by
\begin{equation*}
\gamma_b\left( \sum_{j=2}^n \mu_j \right) = \sum_{j=2}^n \pi^*(\mu_j) \left(\lambda^{j-1} - \frac{(j-1) \pi^* (b_{j-1})}{n}\right),
\end{equation*}
Where $\mu_j \in H^1(\Sigma , K^{1-j})$. Then for any $f \in \mathcal{O}(\mathcal{A})$ and any $m \in h^{-1}(b)$, we have
\begin{equation}\label{ham1}
X_f(m) = \gamma_b( df(b) ).
\end{equation}
\end{proposition}
\begin{proof}
Let $(E,\Phi)$ be a Higgs bundle representing the point $m \in h^{-1}(b)$. Let $(\dot{A} , \dot{\Phi}) \in T_m \mathcal{M}^{{\rm reg}}_{n,d}$ be a tangent vector at $m$. Differentiating the Hitchin map gives $h_*( \dot{A} , \dot{\Phi} ) = ( Tr(\Phi \dot{\Phi} ) , \dots , Tr( \Phi^{n-1}\dot{\Phi} ) )$. Let $f$ be a holomorphic function on $\mathcal{A}$ and set $df(m) = \sum_{j=2}^n [\mu_j]$ where $\mu_j \in \Omega^{0,1}(\Sigma , K^{1-j})$ represents a class $[\mu_j] \in H^1(\Sigma , K^{1-j})$. We find
\begin{equation*}
\begin{aligned}
\Omega_I ( X_f(m) , (\dot{A} , \dot{\Phi} ) ) &= \langle df(m) , h_*( \dot{A} , \dot{\Phi} ) \rangle \\
&= \left\langle \sum_{j=2}^n \mu_j , h_*( \dot{A} , \dot{\Phi} ) \right\rangle \\
&= \sum_{j=2}^n \int_{\Sigma} \mu_j Tr( \Phi^{j-1}\dot{\Phi}) \\
&= \sum_{j=2}^n \int_{\Sigma} Tr( \Phi^{j-1}\mu_j  \dot{\Phi}) \\
&= \Omega_I \left( \left( \sum_{j=2}^n (\Phi^{j-1})_0\mu_j  , 0 \right) , \left( \dot{A} , \dot{\Phi} \right) \right)
\end{aligned}
\end{equation*}
where $(\Phi^{j-1})_0$ denotes the trace-free part of $\Phi^{j-1}$. Note that $Tr( \Phi^{j-1} ) = (j-1)b_{j-1}$, so that $(\Phi^{j-1})_0 = \Phi^{j-1} - \tfrac{(j-1)b_{j-1}}{n}Id$. It follows that $X_f(m)$ is represented by a pair $(\dot{A} , \dot{\Phi})$ of the form $(\dot{A} , \dot{\Phi}) = ( \sum_{j=2}^n (\Phi^{j-1} - \tfrac{(j-1)b_{j-1}}{n}Id) \mu_j , 0 )$. Now suppose that $L$ is the line bundle on $S$ corresponding to $(E,\Phi)$, so $E = \pi_*(L)$ and $\Phi$ is obtained by pushing forward $\lambda : L \to L \otimes \pi^*(K)$. Observe that $\dot{A} = \sum_{j=2}^n (\Phi^{j-1} - \tfrac{(j-1)b_{j-1}}{n}Id) \mu_j $ is obtained by pushing forward $\sum_{j=2}^n \pi^*(\mu_j) (\lambda^{j-1} - \tfrac{(j-1)b_{j-1}}{n} )$, viewed as a deformation of the holomorphic structure on $L$. Equation (\ref{ham1}) follows immediately.
\end{proof}
%%%%%%%%%%%%%%%%%%%%%%%

%%%%%%%%%%%%%%%
% REMARK %%%%%%%%%
%%%%%%%%%%%%%%%
\begin{remark}\label{remhamflow}
If $f$ is any global holomorphic on $\mathcal{A}$, then the corresponding Hamiltonian vector field $X_f$ is defined on all of $\mathcal{M}_{n,d}^{{\rm sm}}$. For any point $m = (\overline{\partial}_E,\Phi) \in \mathcal{M}_{n,d}^{\rm sm}$ we have  
\begin{equation*}
X_f(m) = (\dot{A} , \dot{\Phi}) = \left( \sum_{j=2}^n \mu_j \left(\Phi^{j-1} - \tfrac{(j-1)b_{j-1}}{n}Id\right) , 0\right),
\end{equation*}
where $df(b) = \sum_{j=2}^n \mu_j$ and $b = \pi(m)$. Moreover we can integrate $X_f$ to a biholomorphism $e^{X_f} : \mathcal{M}_{n,d}^{{\rm sm}} \to \mathcal{M}_{n,d}^{{\rm sm}}$ given by
\begin{equation*}
e^{X_f}\left(\overline{\partial}_E , \Phi \right) = \left(\overline{\partial}_E + \sum_{j=2}^n \mu_j \left(\Phi^{j-1} - \tfrac{(j-1)b_{j-1}}{n}Id\right), \Phi \right).
\end{equation*}
The main point to note is that $e^{X_f}$ so defined, preserves semi-stability and preserves the $S$-equivalence relation. This is clear because $(E,\Phi)$ and $e^{X_f}(E,\Phi)$ have the same $\Phi$-invariant sub-bundles. It is also clear that $e^{X_f}$ even extends to an automorphism on the whole of $\mathcal{M}_{n,d}$.
\end{remark}
%%%%%%%%%%%%%%%%%%%%%

%%%%%%%%%%%%%%%
% LEMMA %%%%%%%%%
%%%%%%%%%%%%%%%
\begin{lemma}\label{lemint}
Given $b = (b_2 , \dots , b_n) \in \mathcal{A}^{{\rm reg}}$, let $h_i$ denote the complete homogeneous symmetric polynomials, defined in Equation (\ref{equhpoly}). Then for any $\tau \in \Omega^{0,1}(\Sigma , K^{n-r})$ with $r \ge 0$ we have
\begin{equation*}
\int_S \pi^*(\tau) \lambda^r = \begin{cases} 0 & \text{if } r < n-1, \\  \int_{\Sigma} \tau & \text{if } r = n - 1, \\ \int_{\Sigma} \tau h_{r-n+1} & \text{if } r \ge n. \end{cases}
\end{equation*}
\begin{proof}
Using a partition of unity and the fact that the fibres of $\pi$ containing branch points have measure zero, it suffices to consider the case that $\tau$ is supported in an open set $U \subseteq \Sigma$ which contains no branch points and such that $\pi^{-1}(U) \to U$ is the trivial $n$-fold covering. The derivative $d\pi$ of $\pi$ is a holomorphic section of $K_S \pi^*(K^{-1}) = \pi^*(K^{n-1})$. If the characteristic equation for $S$ is given by $\lambda^{n} + a_2 \lambda^{n-2} + \dots + a_n = 0$ then we claim that 
\begin{equation}\label{dpi}
d\pi = n \lambda^{n-1} + (n-2)a_2 \lambda^{n-3} + \dots + a_{n-1}.
\end{equation}
Indeed both sides of (\ref{dpi}) have the same divisor, so there is a unique isomorphism $K_S \cong \pi^*(K^n)$ for which (\ref{dpi}) holds. Since the covering $\pi^{-1}(U) \to U$ is trivial, the characteristic polynomial may be factored as $(\lambda - \lambda_1) \cdots (\lambda - \lambda_n)$, where the roots $\lambda_i$ are holomorphic sections of $K|_U$. Then
\begin{equation*}
d \pi = \sum_{j=1}^n (\lambda - \lambda_1) \cdots (\widehat{ \lambda - \lambda_i } ) \cdots (\lambda - \lambda_n),
\end{equation*}
where we use $\widehat{ \lambda - \lambda_i }$ to denote omission of the factor $\lambda - \lambda_i$. It follows that
\begin{equation*}
\begin{aligned}
\int_{\pi^{-1}(U)} \pi^*(\tau) \lambda^r &= \int_U \tau \sum_{i=1}^n \frac{ \lambda_i^r}{ d\pi |_{\lambda = \lambda_i} } \\
&= \int_U \tau \sum_{i=1}^n \frac{\lambda_i^r}{ \Pi_{a \neq i} (\lambda_i - \lambda_a) }.
\end{aligned}
\end{equation*}
For any given $u \in U$, we perform a contour integral in the fibre $K_u \cong \mathbb{C}$ over a contour containing all zeros $\lambda_1(u), \lambda_2(u) , \dots , \lambda_n(u)$ of the characteristic polynomial at $u$. For $0 \le r \le n-1$ we find:
\begin{equation*}
\sum_{i=1}^n \frac{\lambda_i^r}{ \Pi_{a \neq i} (\lambda_i - \lambda_a) } = \begin{cases} 0 & \text{if } r < n - 1, \\ 1 & \text{if } r = n-1. \end{cases}
\end{equation*}
This proves the lemma for $0 \le r \le n-1$. The case $r \ge n$ can be proved inductively using the characteristic equation in the form $\lambda^n = -a_2 \lambda^{n-2} - a_3 \lambda^{n-3} - \dots - a_n$. For example when $r = n$, we have
\begin{equation*}
\int_S \pi^*(\tau) \lambda^n = \int_S \pi^*(\tau)( -a_2 \lambda^{n-2} - a_3 \lambda^{n-3} - \dots - a_n ) = 0.
\end{equation*}
In general for $i \ge 0$, we obtain an identity of the form $\int_S \pi^*(\tau) \lambda^{n-1+i} = \int_\Sigma \tau k_i$, where the $k_i$ are given inductively by $k_0 = 1$, $k_i = -\sum_{j=1}^{i} a_j k_{i-j}$, $i \ge 1$. Comparing with (\ref{equnewton}), it follows that the $k_i = h_i$ are the complete homogeneous symmetric polynomials and this proves the lemma.
\end{proof}
\end{lemma}
%%%%%%%%%%%%%

%%%%%%%%%%%%%
% PROPOSITION %%%
%%%%%%%%%%%%%%
\begin{proposition}
Given $b \in \mathcal{A}^{{\rm reg}}$, let $\psi_b : \mathcal{A} \to H^0(S,K_S)$ be the composition
\begin{equation*}
\xymatrix{
\mathcal{A} \ar[rr]^-{(\gamma_b^t)^{-1}} & & H^1(S,\mathcal{O}_S)^* \ar[r]^-{\cong} & H^0(S, K_S) \ar[r]^-{\cong} & H^0( S , \pi^*K^n )
}
\end{equation*} 
where the second arrow is Serre duality. Then for $\nu = \sum_{j=2}^n \nu_j \in \bigoplus_{j=2}^n H^0(\Sigma , K^j) = \mathcal{A}$ we have:
\begin{equation*}
\psi_b(\nu ) = \sum_{j=2}^n \pi^*(\nu_j)( \lambda^{n-j} + a_2 \lambda^{n-j-2} + \dots + a_{n-j} ).
\end{equation*}
\end{proposition}
\begin{proof}
Let $\mu = \sum_{j=2}^n \mu_j \in \bigoplus_{j=2}^n H^1(\Sigma , K^{1-j}) = \mathcal{A}^*$. We need to show for all such $\mu$ that:
\begin{equation}\label{paired}
\int_S \gamma_b(\mu) \psi_b(\nu) = \langle \mu , \nu \rangle = \sum_{j=2}^n \int_{\Sigma} \mu_j \nu_j.
\end{equation}
For this it is convenient to first introduce a map $\theta_b : \mathcal{A} \to H^0(S , K_S) \cong H^0(S, \pi^*K^n)$ which for $\nu_j \in H^0(\Sigma , K^j)$ is given by $\theta_b(\nu_j) = \pi^*(\nu_j) \lambda^{n-j}$. Using Lemma \ref{lemint}, we have:
\begin{equation*}
\int_S \gamma_b(\mu_m) \theta_b(\nu_j) = \begin{cases} 0 & \text{if } m < j, \\ \int_{\Sigma} \mu_j \nu_j & \text{if } m = j, \\ \int_{\Sigma} \mu_m \nu_j h_{m-j}, & \text{if } m > j. \end{cases}
\end{equation*}
From this it follows easily that if we let $\psi_b(\nu_j ) = \pi^*(\nu_j)( \lambda^{n-j} + a_2 \lambda^{n-j-2} + \dots + a_{n-j} )$ then (\ref{paired}) is satisfied.
\end{proof}
%%%%%%%%%%%%%%%%%%%%%%%

%%%%%%%%%%%%%%%%%%%%%%%%%%%%%%%%%%%%%%%%%%%%%%%%

\subsection{Flows on generic singular fibres}\label{sechamflow2}

We will need to examine the Hamiltonian vector fields along the fibres of $h$ lying over generic points of the discriminant divisor $\mathcal{D}$. As noted in \cite{kp}, a generic point $b \in \mathcal{D}$ has a spectral curve $S \to \Sigma$ which has exactly one singular point, which is an ordinary double point. The argument used in \cite{kp} is as follows: let $a_n \in H^0(\Sigma , K^n)$ have a unique double zero $u \in \Sigma$. Then the spectral curve with characteristic equation $\lambda^n + a_n = 0$ has an ordinary double point lying over $u$ and is smooth at all other points. Then by a semicontinuity argument there is a non-empty Zariski open subset $\mathcal{D}^0 \subset \mathcal{D}$ given by
\begin{equation*}
\mathcal{D}^0 = \{ b \in \mathcal{D} | S_b \text{ is irreducible and has a unique ordinary double point} \}.
\end{equation*}
Moreover, the discriminant divisor $\mathcal{D} \subset \mathcal{A}$ is irreducible (\cite[Corollary 1.5]{kp}), so the complement $\mathcal{D} - \mathcal{D}^0$ has positive codimension in $\mathcal{D}$.\\

Consider now the spectral curve $\pi : S \to \Sigma$ associated to a point $b \in \mathcal{D}^0$. Let $p \in S$ be the singular point of $S$ and $u = \pi(p)$. We let $\nu : S^\nu \to S$ be the normalization of $S$ and $\nu^{-1}(p) = \{ p_+ , p_- \}$. Set $\pi^\nu = \pi \circ \nu$. According to \cite{bnr}, the fibre $h^{-1}(b)$ of $\mathcal{M}_{n,d}$ lying over $b$ is a generalised Prym variety
\begin{equation*}
h^{-1}(b) = \overline{Prym}_{d-n(g-1)}(S,\Sigma) = \{ M \in \overline{Jac}_{d-n(g-1)}(S) \, | \, det( \pi_*M ) = L_0 \},
\end{equation*}
where $\overline{Jac}_{d-n(g-1)}(S)$ is the generalised Jacobian consisting of rank $1$, torsion-free sheaves on $S$ with Euler characteristic $d - n(g-1)$. For our purposes it will enough to focus on the dense open subset $\overline{Prym}_{d-n(g-1)}^{\text{loc free}}(S,\Sigma) \subset \overline{Prym}_{d-n(g-1)}(S,\Sigma)$ of the fibre consisting of those $M \in \overline{Prym}_{d-n(g-1)}(S,\Sigma)$ which are locally free. Such $M$ can be equivalently described as follows: start with a holomorphic line bundle $L$ on $S^\nu$ and an isomorphism $\phi : L_{p_+} \to L_{p_-}$. The sheaf of holomorphic sections $s$ of $L$ for which $s(p_-) = \phi( s(p_+) )$ defines a locally-free rank $1$ sheaf $M$ on $S$. Under this correspondence $M \in \overline{Prym}_{d-n(g-1)}(S,\Sigma)$ if and only if $det(\pi^\nu)_* L = L_0(-u)$. If we fix the underlying $\mathcal{C}^\infty$ line bundle $L$ then a point in $\overline{Prym}_{d-n(g-1)}^{\text{loc free}}(S,\Sigma)$ can be represented by a pair $(\overline{\partial}_L , \phi )$, where $\overline{\partial}_L$ is a $\overline{\partial}$-operator on $L$ and $\phi$ is an isomorphism $\phi : L_{p_+} \to L_{p_-}$. A gauge transformation $g : S^\nu \to \mathbb{C}^*$ acts on such pairs as $g( \overline{\partial}_L , \phi ) = (\overline{\partial}_L + g^{-1}dg , g(p_+)g(p_-)^{-1}\phi)$. Then $\overline{Prym}_{d-n(g-1)}^{\text{loc free}}(S,\Sigma)$ is identified with the set of equivalences classes of pairs $(\overline{\partial}_L , \phi)$, subject to the condition that $det (\pi^\nu)_* (L , \overline{\partial}_L) = L_0(-u)$.

%%%%%%%%%%%%%%%%%%%%%%%%
% LEMMA %%%%%%%%%%%%%%
%%%%%%%%%%%%%%%%%%%%%%
\begin{lemma}\label{lemcanon}
The canonical bundle $K_{S^\nu}$ of $S^\nu$ is isomorphic to $(\pi^\nu)^*K^n(-p_+-p_-)$.
\end{lemma}
\begin{proof}
Let $\rho : K_p \to K$ be the blow-up of the total space of $K$ at the point $p$ and $E = \rho^{-1}(p)$ the exceptional divisor. Then $S^\nu \subset K_p$ is the proper transform of $S$. Thus as divisors on $K_p$ we have $S^\nu = \rho^*S - 2E$. Note that $K$ is the cotangent bundle of $\Sigma$, so it has trivial canonical bundle. It follows that the canonical bundle of $K_p$ is $[E]$ and by adjunction we have $K_{S^\nu} = ([S^\nu] + [E])|_{S^\nu} = (\rho^*[S] - [E])|_{S^\nu} = (\pi^\nu)^*K^n - [p_+ + p_-]$. That is, $K_{S^\nu} \cong (\pi^\nu)^*K^n(-p_+-p_-)$.
\end{proof}
%%%%%%%%%%%%%%%%%%%%%%%

%%%%%%%%%%%%%%%%%%%%
% PROPOSITION %%%%%%%%%%
%%%%%%%%%%%%%%%%%%%
\begin{proposition}\label{prophitsurj}
For any $b \in \mathcal{D}^0$, the differential $h_*$ of the Hitchin map is surjective at all points of $\overline{Prym}_{d-n(g-1)}^{\text{loc free}}(S,\Sigma)$.
\end{proposition}
\begin{proof}
First note that $\overline{Prym}_{d-n(g-1)}(S,\Sigma) \subset \mathcal{M}_{n,d}^{{\rm sm}}$. Indeed it can be shown that every point in $\overline{Prym}_{d-n(g-1)}(S,\Sigma)$ gives rise to a stable Higgs bundle \cite[Remark 1.5]{kp}. Any $f \in \mathcal{A}^*$ may be thought of as a linear function on $\mathcal{A}$, hence has an associated Hamiltonian vector field $X_f$. To show that $h_*$ is surjective on $\overline{Prym}_{d-n(g-1)}^{\text{loc free}}(S,\Sigma)$ it is clearly sufficient to show that for any non-zero $f \in \mathcal{A}^*$, $X_f$ is non-vanishing on $\overline{Prym}_{d-n(g-1)}^{\text{loc free}}(S,\Sigma)$. To show this, we need to examine how the Hamiltonian flow of $X_f$ acts on pairs $(\overline{\partial}_L , \phi)$. Given such a pair $(\overline{\partial}_L , \phi )$, let $F = (\pi^\nu)_* L$ and let $\Phi' : F \to F \otimes K$ be the endomorphism obtained by pushing forward $\nu^*(\lambda)$. The Higgs bundle $(E,\Phi)$ associated to $(\overline{\partial}_L , \phi)$ is related to $(F , \Phi')$ through a Hecke modification as follows: the isomorphism $\phi : L_{p_+} \to L_{p_-}$ defines a codimension $1$ subspace $F_\phi \subset F_u$ of the fibre of $F$ at $u$. Then $\mathcal{O}(E)$ is defined as the subsheaf of $\mathcal{O}(F)$ consisting of sections $s$ for which $s(u) \in F_{\phi}$ (whenever $u$ is in the domain of $s$). By construction $\Phi'$ preserves the subspace $V_\phi$ and the Higgs field $\Phi$ is simply the restriction of $\Phi'$ to $\mathcal{O}(E)$.\\

Writing $df = \sum_{j=2}^n f_j$, for $f_j \in \Omega^{0,1}(\Sigma , K^{1-j})$, we have by Remark \ref{remhamflow} that the flow of $X_f$ on $(\overline{\partial}_E , \Phi)$ is given by 
\begin{equation*}
e^{tX_f}(\overline{\partial}_E , \Phi ) = \left(\overline{\partial}_E + t\sum_{j=2}^n f_j \left( \Phi^{j-1} - \frac{(j-1)b_{j-1}}{n} Id\right) , \Phi \right).
\end{equation*}
At the level of pairs $(\overline{\partial}_L , \phi )$ it follows that the flow is given by
\begin{equation*}
\left( \overline{\partial}_L , \phi \right) \mapsto \left( \overline{\partial}_L + t \sum_{j=2}^n (\pi^\nu)^*(f_j) \left( \nu^*(\lambda^{j-1}) - \frac{(j-1)(\pi^\nu)^* (b_{j-1})}{n} \right) , \phi \right).
\end{equation*}
Now suppose that $f \in \mathcal{A}^*$ is non-zero and that $X_f$ vanishes at $(\overline{\partial}_L , \phi)$. This is equivalent to requiring that
\begin{equation}\label{equfis}
\sum_{j=2}^n (\pi^\nu)^*(f_j) \left( \nu^*(\lambda^{j-1}) - \frac{(j-1)(\pi^\nu)^*(b_{j-1})}{n} \right) = \overline{\partial} g,
\end{equation}
for some $g : S^\nu \to \mathbb{C}$ satisfying $g(p_+) = g(p_-)$. For any $\mu_j \in \Omega^{0,1}(\Sigma , K^{1-j})$ let us define
\begin{equation*}
\tilde{\gamma}_b \left( \sum_{j=2}^n \mu_j \right) = \sum_{j=2}^n (\pi^\nu)^*(\mu_j) \left( \nu^*(\lambda^{j-1}) - \frac{(j-1)(\pi^\nu)^*b_{j-1}}{n} \right) \in \Omega^{0,1}(S^\nu).
\end{equation*}
Furthermore, let us define $\tilde{\psi}_b : \mathcal{A} \to H^0(S^\nu , (\pi^\nu)^*(K^n) )$ by
\begin{equation*}
\tilde{\psi}_b( \nu_j ) = (\pi^\nu)^*(\nu_j)( \nu^*\lambda^{n-j} + a_2 \nu^*\lambda^{n-j-2} + \dots + a_{n-j} ).
\end{equation*}
By Lemma \ref{lemcanon}, $(\pi^\nu)^*K^n \cong K_{S^\nu}(p_+ + p_-)$. Thus we may interpret $H^0(S^\nu , (\pi^\nu)^*(K^n) )$ as the space of meromorphic sections of $K_{S^\nu}$ which have at worst first order poles at $p_+,p_-$. For any $\mu_j \in \Omega^{0,1}(\Sigma , K^{1-j})$ and $\nu_m \in H^0( \Sigma , K^m)$ we have that $\tilde{\gamma}_b(\mu_j) \tilde{\psi}_b(\nu_m)$ can be viewed as a $(1,1)$-form on $S^\nu$ away from the poles of $\tilde{\psi}_b(\nu_m)$. Now as the poles of $\tilde{\psi}_b(\nu_m)$ are first order it is easy to see that $\tilde{\gamma}_b(\mu_j) \tilde{\psi}_b(\nu_m)$ is integrable on $S^\nu$. Moreover, since $S^\nu$ coincides with $S$ away from a set of measure zero we find that
\begin{equation*}
\int_{S^\nu} \tilde{\gamma}_b(\mu_j) \tilde{\psi}_b(\nu_m) = \int_S \gamma_b(\mu_j) \psi_b(\nu_m) = \langle \mu_j , \nu_m \rangle,
\end{equation*}
where as usual, $\langle \, \, , \, \, \rangle$ is the pairing of $\mathcal{A}^*$ and $\mathcal{A}$. From Equation (\ref{equfis}) we have that $\tilde{\gamma}_b(df) = \overline{\partial} g$, where $g$ is a function on $S^\nu$ such that $g(p_+) = g(p_-)$. Thus, for any $\nu_m \in H^0( \Sigma , K^m)$ we have that:
\begin{equation*}
\langle df , \nu_m \rangle = \int_{S^\nu} (\overline{\partial} g) \tilde{\psi}_b(\nu_m).
\end{equation*}
If $\tilde{\psi}_b(\nu_m)$ has no poles then this expression vanishes. More generally, suppose that $\tilde{\psi}_b(\nu_m)$ has first order poles at $p_+,p_-$ with residues $r_+,r_-$. We have $r_+ + r_- = 0$, since the residues of a meromorphic $1$-form on $S^\nu$ must sum to zero. Choose local coordinates $z_+,z_-$ centered at $p_+,p_-$ and let $D_+(\epsilon),D_-(\epsilon)$ be the corresponding discs of radius $\epsilon$ around $p_+,p_-$. Then
\begin{equation*}
\begin{aligned}
\int_{S^\nu} (\overline{\partial} g) \tilde{\psi}_b(\nu_m) &= \lim_{\epsilon \to 0} \int_{S^\nu - D_+(\epsilon) - D_-(\epsilon)} (\overline{\partial} g) \tilde{\psi}_b(\nu_m) \\
&= \lim_{\epsilon \to 0} \int_{S^\nu - D_+(\epsilon) - D_-(\epsilon)} d \left( g \tilde{\psi}_b(\nu_m) \right) \\
&= \lim_{\epsilon \to 0} \left( \int_{\partial D_+(\epsilon)} g \tilde{\psi}_b(\nu_m) + \int_{\partial D_-(\epsilon)} g \tilde{\psi}_b(\nu_m) \right) \\
&= 2 \pi i ( g(p_+)r_+ + g(p_-)r_-).
\end{aligned}
\end{equation*}
But $g(p_+) = g(p_-)$, so this is $2 \pi i g(p_+)( r_+ + r_- ) = 0$. We have shown that $\langle df , \nu_m \rangle = 0$ for all $\nu_m$, hence $df = 0$. But $f$ is a linear function so this means $f = 0$. But we assumed that $f$ is non-zero, hence $X_f$ must be non-vanishing.
\end{proof}
%%%%%%%%%%%%%%%%%%%%%%%%%%

%%%%%%%%%%%%%%%%%%%%%
% COROLLARY %%%%%%%%%%%%
%%%%%%%%%%%%%%%%%%%%%
\begin{corollary}\label{corcodim}
The set of points of $\mathcal{M}_{n,d}^{{\rm sm}}$ where the differential of the Hitchin map is not surjective has codimension $\ge 2$.
\end{corollary}

%%%%%%%%%%%%%%%%%%%%%%%%%%%%%%%%%%%%%%%%%%%%%%%%%%%%%%%%%%%%%%%%%%%%%%%%%%%%
%%%%%%%%%%%%%%%%%%%%%%%%%%%%%%%%%%%%%%%%%%%%%%%%%%%%%%%%%%%%%%%%%%%%%%%%%%%%
\section{Holomorphic vector fields on $\mathcal{M}_{n,d}^{{\rm sm}}$}\label{secholovf}

%%%%%%%%%%%%%%%%%%%%%%%%%%%%%%%%%%%%%%%%

\subsection{Kodaira-Spencer maps}\label{secksm}

%%%%%%%%%%%%%%%%%%%%%%%%%%%%%
% LEMMA %%%%%%%%%%%%%%%%%%%%%%%
%%%%%%%%%%%%%%%%%%%%%%%%%%%%%
\begin{lemma}\label{lemmahyper}
Let $S \to \Sigma$ be a non-singular spectral curve, where $\Sigma$ has genus $g > 1$. Then $S$ is not hyperelliptic.
\end{lemma}
\begin{proof}
Suppose on the contrary, that $S$ is hyperelliptic. Let $\sigma : S \to S$ be the hyperelliptic involution. Then $\sigma$ acts on $H^0(S , K_S ) \cong H^0(\Sigma , \pi^* K^n)$ as multiplication by $-1$. As usual, we let $\lambda$ denote the tautological section of $\pi^* K$. Then $s = \sigma^*(\lambda)$ is a section of $\sigma^*( \pi^* K)$. It follows that $s^n = \sigma^* (\lambda^n ) = -\lambda^n$. Thus $s^n$ and $\lambda^n$ have the same divisor. Of course this means $s$ and $\lambda$ also have the same divisor and $\sigma^* \pi^* K \cong \pi^* K$. It follows that we can lift $\sigma$ to an isomorphism $\hat{\sigma} : \pi^* K \to \pi^* K$ covering $\sigma$ and for which $\hat{\sigma}^{\otimes n} = \sigma^*$ is the pullback by $\sigma$ on $\pi^* K^n \cong K_S$. But as $\sigma$ is an involution, we have $\sigma^* \circ \sigma^* = 1$ and hence $\hat{\sigma} \circ \hat{\sigma} = c \, Id$ for some constant $c$ satisfying $c^n = 1$.\\

From the identity $\pi_* \mathcal{O}_S = \mathcal{O}_\Sigma \oplus \mathcal{O}(K^{-1}) \oplus \dots \oplus \mathcal{O}(K^{-n+1})$ \cite{bnr}, we see that $H^0(S , \pi^* K) \cong H^0(\Sigma , \mathcal{O}) \oplus H^0( \Sigma , K) \cong \mathbb{C} \oplus H^0( \Sigma , K)$. Thus $\hat{\sigma}(\lambda) = a\lambda + \pi^*(b)$, where $a \in \mathbb{C}$ and $b \in H^0(\Sigma , K)$. But now we have $-\lambda^n = \sigma^*(\lambda^n) = (a \lambda + b)^n$ and so $\lambda^n + (a \lambda + b )^n = 0$. Note that $\lambda^n + (a \lambda + b )^n$ factors as a polynomial in $\lambda$. This would contradict irreducibility of $S$ unless the coefficients of this polynomial all vanish, so $a^n = -1$ and $b = 0$. Therefore $\hat{\sigma}(\lambda) = a\lambda$ and since $\hat{\sigma} \circ \hat{\sigma} = c \, Id$, we need $a^2 = c$. Note in particular that $a^{2n} = 1$.\\

Now let $\alpha \in H^0( \Sigma , K^j)$. Then $\lambda^{(2n-1)j} \pi^* \alpha \in H^0( S , \pi^* K^{2nj} ) \cong H^0( S , K_S^{2j})$. Therefore $\sigma^*( \lambda^{(2n-1)j} \pi^* \alpha ) = \lambda^{(2n-1)j} \pi^* \alpha$ and so $a^{(2n-1)j} \hat{\sigma}^{\otimes j}( \pi^* \alpha ) = \alpha$. Hence using $a^{2n} = 1$, we have $\hat{\sigma}^{\otimes j}( \pi^* \alpha ) = a^j \pi^* \alpha$. Next, using $\mathcal{O}_S = \mathcal{O}_\Sigma \oplus \mathcal{O}(K^{-1}) \oplus \dots \oplus \mathcal{O}(K^{-n+1})$, we have $H^0(S , K_S^2 ) \cong H^0( \Sigma , K^{2n} ) \oplus \dots \oplus H^0( \Sigma , K^{n+1})$ and any $\omega \in H^0( S , K_S^2)$ can be written as $\omega = \pi^* \omega_{2n} + \dots + \pi^* \omega_{n+1}\lambda^{n-1}$, where $\omega_j \in H^0( \Sigma , K^j)$. It follows that $\sigma^*( \omega ) = a^{2n} \omega = \omega$, so that $\sigma$ acts as the identity on $H^0( S , K_S^2)$. However this never happens for a hyperelliptic curve of genus $> 2$. On the other hand the genus of $S$ satisfies $g_S = 1 + n^2(g-1) > 2$, so this is a contradiction.
\end{proof}
%%%%%%%%%%%%%%%%%%%%%%%%%%%%%%%

Recall that $\mathcal{M}_{n,d}$ admits an action of $\mathbb{C}^*$ as follows: for any $\lambda \in \mathbb{C}^*$ we let $m_\lambda : \mathcal{M}_{n,d} \to \mathcal{M}_{n,d}$ be defined as $m_\lambda( E , \Phi ) = (E , \lambda \Phi )$. There is a unique $\mathbb{C}^*$-action $m^\mathcal{A}_\lambda$ on $\mathcal{A}$ compatible with the $\mathbb{C}^*$-action on $\mathcal{M}_{n,d}$ in the sense that $h \circ m_\lambda = m^\mathcal{A}_\lambda \circ h$. Note that this $\mathbb{C}^*$-action preserves $\mathcal{A}^{\rm reg}$. Under the decomposition $\mathcal{A} = \bigoplus_{j=2}^n H^0(\Sigma , K^j)$ we have that the subspace $H^0(\Sigma , K^j)$ has weight $j$ with respect to this action. Let $\xi = \left. \tfrac{d}{dt} \right|_{t=0} m_{e^t}$ ($t \in \mathbb{R}$) be the holomorphic vector field on $\mathcal{M}^{\rm sm}_{n,d}$ tangent to the action of $\mathbb{R}_{+} \subset \mathbb{C}^*$ and similarly let $\xi^\mathcal{A} = \left. \tfrac{d}{dt}\right|_{t=0} m^\mathcal{A}_{e^t}$. It follows that $h_* \xi_m = \xi^\mathcal{A}_{h(m)}$ for all $m \in \mathcal{M}_{n,d}^{\rm sm}$.\\

Recall that to any point $b \in \mathcal{A}$ we associate a section $s_b$ of $\pi^*(K^n)$ on the total space of $K$ and that the corresponding spectral curve $S_b$ is the zero locus of $s_b$. We can similarly construct the universal family $S^{\rm reg}_{\rm univ}$ of regular spectral curves:
\begin{equation*}
S^{\rm reg}_{\rm univ} = \{ (b,u) \in \mathcal{A}^{\rm reg} \times K | s_b(u) = 0 \}.
\end{equation*}
Let $q : S^{\rm reg}_{\rm univ} \to \mathcal{A}^{\rm reg}$ be the projection $(b,u) \mapsto b$. It is clear that $S^{\rm reg}_{\rm univ}$ is smooth and that the fibre of $q$ over $b$ is precisely the spectral curve $S_b$. Thus for any $b \in \mathcal{A}^{\rm reg}$ we have a Kodaira-Spencer map $\rho_b : \mathcal{A} \cong T_b \mathcal{A}^{\rm reg} \to H^1( S_b , TS_b)$.

%%%%%%%%%%%%%%%%%%%%%%%%%
% LEMMA %%%
%%%%%%%%%%%%%%%%%%%%%%%%%
\begin{lemma}\label{lemksker1}
For any $b \in \mathcal{A}^{\rm reg}$, the kernel of the Kodaira-Spencer map $\rho_b : \mathcal{A} \to H^1( S_b , TS_b )$ is spanned by $\xi^\mathcal{A}_b$.
\end{lemma}
\begin{proof}
Denote $S_b$ more simply as $S$. On $S$ we have an exact sequence
\begin{equation}\label{equles1}
0 \to TS \to TK|_S \buildrel \beta \over \to N \to 0,
\end{equation}
where $N$ is the normal bundle to $S$ in $K$. The Kodaira-Spencer map fits into the following commutative diagram
\begin{equation*}
\xymatrix{
0 \ar[r] & H^0(S , TK|_S) \ar[r]^\beta & H^0(S , N ) \ar[r]^-{\delta} & H^1(S , TS ) \\
& & \mathcal{A} \ar[u]^{\chi} \ar[ur]^{\rho_b} &
}
\end{equation*}
where the horizontal sequence of maps is obtained from the long exact sequence associated to (\ref{equles1}) and $\chi : \mathcal{A} \to H^0(S , N)$ is the characteristic map \cite{gri}. We have $N = [S]|_S = \pi^*K^n$, hence $H^0(S,N) = H^0(S,\pi^*K^n) = \bigoplus_{j=1}^n H^0(\Sigma , K^j)$. One sees that the map $\chi$ is the obvious inclusion $\mathcal{A} = \bigoplus_{j=2}^n H^0(\Sigma , K^j) \subset \bigoplus_{j=1}^n H^0(\Sigma , K^j)$. By exactness of the horizontal sequence we have that $Ker(\rho_b) = Ker(\delta) \cap Im(\chi) = Im(\beta) \cap Im(\chi)$.\\

Observe that on the total space of $K$ there is an exact sequence $0 \to \pi^*K \to TK \to \pi^*K^{-1} \to 0$. Restricting to $S$ and taking the associated long exact sequence gives:
\begin{equation*}
0 \to H^0( S , \pi^* K ) \to H^0( S , TK|_S ) \to H^0(S , \pi^*K^{-1}).
\end{equation*}
Now using $\pi_* \mathcal{O}_S = \mathcal{O}_\Sigma \oplus \mathcal{O}(K^{-1}) \oplus \dots \oplus \mathcal{O}(K^{-n+1})$, we see that $H^0(S , \pi^* K^{-1}) = 0$. Therefore we have an isomorphism $H^0(S , TK|_S ) \cong H^0(S , \pi^*K)$. Moreover, we have $H^0(S , \pi^* K) \cong H^0(\Sigma , \mathcal{O} ) \oplus H^0(\Sigma , K)$. In fact, it is easy to see what the corresponding sections on $H^0(S , TK|_S)$ are. The factor $H^0(\Sigma , \mathcal{O}) \cong \mathbb{C}$ is spanned by the vector field generating the $\mathbb{C}^*$-action in the fibres of $K \to \Sigma$. An element $\alpha \in H^0(\Sigma , K)$ defines a section of $TK|_S$ whose value at $s \in S$ is $\alpha( \pi(s) ) \in (\pi^*K)_s \subset (TK)_s$. Given an element $(c , \alpha ) \in H^0(\Sigma , \mathcal{O}) \oplus H^0(\Sigma , K)$ it is easy to see that $\beta( c , \alpha )$ is in the image of $\chi$ only if $\alpha = 0$ (because our spectral curves have no $a_1$ coefficient in their characteristic polynomial). This leaves a $1$-dimensional space of deformations of $S$ in $K$ generated by the $\mathbb{C}^*$-action in the fibres of $K$. As this corresponds to the $\mathbb{C}^*$-action on $\mathcal{A}$, we have shown that the kernel of $\rho_b$ is indeed spanned by $\xi^\mathcal{A}_b$.
\end{proof}
%%%%%%%%%%%%%%%%%%%%%%%%%

In a similar fashion we can view $h: \mathcal{M}_{n,d}^{\rm reg} \to \mathcal{A}^{\rm reg}$ as a family of abelian varieties, hence to any $b \in \mathcal{A}^{\rm reg}$ we have a Kodaira-Spencer map $\theta_b : \mathcal{A} \to H^1( h^{-1}(b) , T h^{-1}(b) )$.

%%%%%%%%%%%%%%%%%%%%%%%%%
% LEMMA %%%
%%%%%%%%%%%%%%%%%%%%%%%%%
\begin{lemma}\label{lemksker2}
Let $Y$ be a holomorphic vector field on $\mathcal{A}^{\rm reg}$ such that $\theta_b(Y_b) = 0$ for all $b \in \mathcal{A}^{\rm reg}$. Then $Y = f \xi^\mathcal{A}$ for some holomorphic function $f$ on $\mathcal{A}^{\rm reg}$.
\end{lemma}
\begin{proof}
Let $Jac_{\tilde{d}}(S)$ be the degree $\tilde{d}$ component of the Picard variety of $S$ and $h^{-1}(b) = Prym_{\tilde{d}}(S,\Sigma)$ the Prym variety. We can in the same way define a Kodaira-Spencer map $\tau_b : \mathcal{A} \to H^1( Jac(S) , TJac(S) )$. Let $Z = Prym_{\tilde{d}} \times Jac(\Sigma)$. The map $p : Z \to Jac_{\tilde{d}}(S)$ given by $p(M,N) = M \otimes \pi^*(N)$ is a covering space with fibre $Jac(\Sigma)[n]$, the points of order $n$ in $Jac(\Sigma)$. Furthermore $TZ \cong p^*(TJac_{\tilde{d}}(S)) \simeq \mathbb{C}^{2g_S}$ and by Hodge theory the pullback $p^* : H^1( Jac_{\tilde{d}}(S) , TJac_{\tilde{d}}(S) ) \to H^1( Z , TZ )$ is an isomorphism. It follows that if $\theta_b(Y_b) = 0$ then the corresponding deformation of $Jac_{\tilde{d}}(S)$ is trivial and so is the deformation of $Jac(S)$. This means that $\tau_b(Y_b) = 0$. Next observe that there is a natural map $J : H^1(S , TS ) \to H^1( Jac(S) , TJac(S) )$ such that $\tau_b = J \circ \rho_b$. By Lemma \ref{lemmahyper}, $S$ is not hyperelliptic and as is well known, this implies that $J$ is injective. Therefore $\tau_b(Y_b) = 0$ implies that $\rho_b(Y_b) = 0$ and by Lemma \ref{lemksker1}, $Y$ is a multiple of $\xi^\mathcal{A}_b$. 
\end{proof}
%%%%%%%%%%%%%%%%%%%%%%%%%

%%%%%%%%%%%%%%%%%%%%%%%%%%%%%%%%%%%%%%%%%%%%%%%%%

\subsection{Classification of holomorphic vector fields}\label{secchvf}

The restriction $h : \mathcal{M}_{n,d}^{{\rm reg}} \to \mathcal{A}^{\rm reg}$ of the Hitchin fibration over $\mathcal{A}^{\rm reg}$ is a smooth fibre bundle, so we have an exact sequence:
\begin{equation*}
\xymatrix{
0 \ar[r] & {\rm Ker}( h_* ) \ar[r] & T \mathcal{M}_{n,d}^{{\rm reg}} \ar[r]^-{h_*} & \mathcal{A} \ar[r] & 0,
}
\end{equation*}
where $\mathcal{A}$ is thought of as a trivial vector bundle on $\mathcal{M}_{n,d}^{{\rm reg}}$. The map sending a linear function $f \in \mathcal{A}^*$ to the corresponding Hamiltonian vector field $X_f$ gives an isomorphism ${\rm Ker}(h_*) \cong \mathcal{A}^*$, hence our exact sequence becomes:
\begin{equation}\label{equextang}
\xymatrix{
0 \ar[r] & \mathcal{A}^* \ar[r] & T \mathcal{M}_{n,d}^{{\rm reg}} \ar[r]^-{h_*} & \mathcal{A} \ar[r] & 0.
}
\end{equation}

%%%%%%%%%%%%%%%%%%%%%%%%
% PROPOSITION
%%%%%%%%%%%%%%%%%%%%%%%%
\begin{proposition}\label{propxbase}
Let $X$ be a holomorphic vector field on $\mathcal{M}_{n,d}^{\rm reg}$. There is a holomorphic function $f$ on $\mathcal{A}^{\rm reg}$ such that for any $m \in \mathcal{M}_{n,d}^{\rm reg}$ we have $h_* X_m = f(b) \xi^\mathcal{A}_b$, where $b = h(m)$.
\end{proposition}
\begin{proof}
From (\ref{equextang}) we see that $h_*X$ is a section of the trivial bundle $\mathcal{A}$ and therefore must be constant over the fibres of $h$. So there is a holomorphic vector field $Y$ on $\mathcal{A}^{\rm reg}$ such that $h_* X_m = Y_{h(m)}$ for all $m \in \mathcal{M}_{n,d}^{\rm reg}$. We need to show that $Y = f\xi^\mathcal{A}$ for some holomorphic function $f$ on $\mathcal{A}^{\rm reg}$.\\

Let $b \in \mathcal{A}^{\rm reg}$ and let $\epsilon > 0$ be such that there exists an integral curve $\rho_b(t) : (-\epsilon , \epsilon ) \to \mathcal{A}^{\rm reg}$ of $Y$ with $\gamma_b(0) = b$. The integral curves of $X$ through points of $h^{-1}(b)$ must lie over $\gamma_b(t)$. This together with properness of the Hitchin map ensures that for every $m \in h^{-1}(b)$, the integral curve $\hat{\gamma}_m(t)$ of $X$ with $\hat{\gamma}_m(0) = m$ exists for $t$ in the interval $(-\epsilon , \epsilon)$. The fact that $X$ is holomorphic now implies that the fibres $h^{-1}( \gamma_b(t) )$ over $\gamma_b(t)$ are all biholomorphic to $h^{-1}(b)$. Therefore $Y_b$ is in the Kernel of the Kodaira-Spencer map $\theta_b$. Hence by Lemma \ref{lemksker2}, $Y$ has the expected form.
\end{proof}
%%%%%%%%%%%%%%%%%%%%%%

We now give a construction of a large family of holomorphic vector fields on $\mathcal{M}_{n,d}^{\rm reg}$. Let $\mu$ be a holomorphic $1$-form on $\mathcal{A}^{\rm reg}$. Then we define a holomorphic vector field $X_\mu$ on $\mathcal{M}_{n,d}^{\rm reg}$ by the relation
\begin{equation*}
i_{X_\mu} \Omega_I = h^*\mu.
\end{equation*}
Note that if $\mu$ extends to a holomorphic $1$-form on all of $\mathcal{A}$, then $X_\mu$ extends to all of $\mathcal{M}_{n,d}^{\rm sm}$.

%%%%%%%%%%%%%%%%%%%%%%%
% THEOREM %%%%%%%%%%%%%
%%%%%%%%%%%%%%%%%%%%%%%
\begin{theorem}\label{theoremvecreg}
Let $X$ be a holomorphic vector field on $\mathcal{M}_{n,d}^{\rm reg}$. Then there exists a holomorphic function $f$ and holomorphic $1$-form $\mu$ on $\mathcal{A}^{\rm reg}$ such that 
\begin{equation*}
X = h^*(f) \xi + X_\mu.
\end{equation*}
Conversely, for every such $f,\mu$ we obtain a holomorphic vector field $X = h^*(f)\xi + X_\mu$. Moreover $X$ extends to $\mathcal{M}_{n,d}^{\rm sm}$ if and only if $f$ and $\mu$ extend to $\mathcal{A}$.
\end{theorem}
\begin{proof}
Let $X$ be a holomorphic vector field on $\mathcal{M}_{n,d}^{\rm reg}$. By Proposition \ref{propxbase}, there is a holomorphic function $f$ on $\mathcal{A}^{\rm reg}$ such that $h_*(X_m) = f(b) \xi^\mathcal{A}_b$, where $b = h(m)$. Consider the holomorphic vector field $Y = X - h^*(f)\xi$ on $\mathcal{M}_{n,d}^{\rm reg}$. By construction $h_*(Y) = 0$, so $Y$ is valued in $Ker(h_*)$, which is a trivial bundle isomorphic to $\mathcal{A}^*$. It follows that $Y$ is constant along the fibres of $h$, hence $Y = X_\mu$ for a unique holomorphic $1$-form on $\mathcal{A}^{\rm reg}$. This shows that $X = h^*(f) \xi + X_\mu$ as required. It is clear that if $f$ and $\mu$ extend to $\mathcal{A}$, then $X$ extends to $\mathcal{M}_{n,d}^{\rm sm}$.\\

Conversely, suppose that $X$ extends to $\mathcal{M}_{n,d}^{\rm sm}$. It follows that $f$ extends to $\mathcal{A}$, because $h : \mathcal{M}_{n,d}^{\rm sm} \to \mathcal{A}$ is surjective, and $h_*( X_m ) = f(b) \xi^\mathcal{A}_b$ for any $m \in \mathcal{M}_{n,d}^{\rm reg}$ and $b = h(m)$. Let $U \subset \mathcal{M}_{n,d}^{\rm sm}$ be the points where $h_*$ is surjective and set $Y = X - h^*(f)\xi$. Since $X$ and $h^*(f)$ extend to $\mathcal{M}_{n,d}^{\rm sm}$, $Y$ also extends to $\mathcal{M}_{n,d}^{\rm sm}$ and satisfies $h_*(Y) = 0$. It follows that there is a holomorphic $\mathcal{A}^*$-valued function $\alpha$ on $U$ such $i_{Y_m} (\Omega_I)_m = h^*( \alpha(m) )$ for any $m \in U$. In particular this means that $\alpha$ is an extension of $h^*(\mu)$ from $\mathcal{M}_{n,d}^{\rm reg}$ to $U$. Now suppose that $b$ belongs to $\mathcal{D}^0$. By Proposition \ref{prophitsurj}, there is a point $m \in h^{-1}(b)$ such that $h_*$ is surjective at $m$. Note that since $Y$ commutes with the Hamiltonian flows, we may find an open neighborhood $N \subseteq \mathcal{A}$ with $b \in N$ and a holomorphic $\mathcal{A}^*$-valued function $\mu'$ on $N$ such that $\alpha = h^*(\mu')$ in a neighborhood of $m$. Clearly $\mu'$ must agree with $\mu$ on $N \cap (\mathcal{A} - \mathcal{D})$. Since this happens for all $b \in \mathcal{D}^0$, we have shown that $\mu$ extends to $\mathcal{A} - \mathcal{E}$, where $\mathcal{E} = \mathcal{D} - \mathcal{D}^0$. But $\mathcal{E}$ has codimension at least $2$ in $\mathcal{A}$, so in fact $\mu$ extends to the whole of $\mathcal{A}$.
\end{proof}
%%%%%%%%%%%%%%%%%%%%%%%

%%%%%%%%%%%%%%%%%%%%%
% DEFINITION %%%%%%%%%
%%%%%%%%%%%%%%%%%%%%%
\begin{definition}
By a {\em $1$-parameter subgroup} $\phi_t$ of automorphisms of $\mathcal{M}_{n,d}$, we mean a group homomorphism $(\mathbb{R} , + ) \to Aut(\mathcal{M}_{n,d}), t \mapsto \phi_t$. If $X$ is a holomorphic vector field on $\mathcal{M}_{n,d}^{\rm sm}$, we say that $X$ {\em integrates} to $\phi_t$ if $X$ integrates to the restriction of $\phi_t$ to $\mathcal{M}_{n,d}^{\rm sm}$ in the usual sense.
\end{definition}
%%%%%%%%%%%%%%%%%%%

%%%%%%%%%%%%%%%%
% PROPOSITION
%%%%%%%%%%%%%%%%
\begin{proposition}\label{propintegrate}
Let $X$ be a holomorphic vector field on $\mathcal{M}_{n,d}^{\rm sm}$, which by Theorem \ref{theoremvecreg} can be written in the form $X = h^*(f)\xi + X_\mu$ with $f$ a holomorphic function on $\mathcal{A}$ and $\mu$ a holomorphic $1$-form on $\mathcal{A}$. Then $X$ integrates to a $1$-parameter subgroup $\phi_t$ of $Aut(\mathcal{M}_{n,d})$ if and only if $f$ is constant.
\end{proposition}
\begin{proof}
Suppose that $X = h^*(f)\xi + X_\mu$ integrates to a $1$-parameter subgroup $\phi_t$ of automorphisms of $\mathcal{M}_{n,d}$. Since the only global holomorphic functions on $\mathcal{M}_{n,d}$ are those of the form $h^*(g)$, where $g$ is a holomorphic function on $\mathcal{A}$, we see that there is a $1$-parameter subgroup $\psi_t$ of automorphisms of $\mathcal{A}$ such that $h \circ \phi_t = \psi_t \circ h$. It follows that $f\xi^\mathcal{A}$ integrates to $\psi_t$.

Choose a point $b \in \mathcal{A}$ for which $\mathbb{C}^*$  acts freely (or with kernel $\pm 1$ in the case $n = 2$). Let $\mathcal{O}_b \cong \mathbb{C}^*$ be the orbit. Then $\psi_t$ restricts to a $1$-parameter family of automorphisms of $\mathcal{O}_b$. Any automorphism of $\mathbb{C}^*$ is either of the form $z \mapsto cz$ or $z \mapsto cz^{-1}$, where $c \in \mathbb{C}^*$. In our case $\psi_t$ is connected to the identity, so the automorphisms $\psi_t$ of $\mathcal{O}_b$ must be of the form $z \mapsto m_{e^{at}} z$ for some $a \in \mathbb{C}$. Thus for any $b \in \mathcal{A}$ for which the stabiliser of the $\mathbb{C}^*$-action is trivial (or $\pm 1$ in the case $n=2$) we have deduced that $\psi_t(b) = m_{e^{at}} b$, for some $a \in \mathbb{C}$. In fact it is clear that $a = f(b)$. What we have shown is that $f$ is constant on the orbit $\mathcal{O}_b$. Since the closure of any such orbit contains the origin $0 \in \mathcal{A}$, we see that $f$ must be a constant.\\

Conversely, suppose that $X = f\xi + X_\mu$, where $f$ is constant. In the case that $f = 0$, it is easy to integrate $X$. In fact if $(\overline{\partial}_E,\Phi) \in \mathcal{M}_{n,d}$ and $b = h(\overline{\partial}_E , \Phi )$ then
\begin{equation*}
e^{tX}(\overline{\partial}_E , \Phi ) = \left(\overline{\partial}_E + t \sum_{j=2}^n \mu_j(b) ( \Phi^{j-1} - \frac{(j-1)b_{j-1}}{n} Id ) , \Phi \right),
\end{equation*}
where $\mu_j(b)$ is the component of $\mu(b)$ in $H^1(\Sigma , K^{1-j})$. More generally, when $f$ is non-zero we find, 
\begin{equation*}
e^{tX}(\overline{\partial}_E , \Phi ) = \left(\overline{\partial}_E + \sum_{j=2}^n \alpha_{j,t}(b) ( \Phi^{j-1} - \frac{(j-1)b_{j-1}}{n} Id ) , e^{tf} \Phi \right),
\end{equation*}
where $\alpha_{j,t}(b)$ is given by
\begin{equation*}
\alpha_{j,t}(b) = \int_0^t e^{(j-1)uf} \mu_j( e^{uf}b )du.
\end{equation*}

\end{proof}
%%%%%%%%%%%%%%%%%

%%%%%%%%%%%%%%%%%%%%%%%%%%%%%%%%%%%%%%%%%%%%%%%%%%%%55

\subsection{Further properties of the holomorphic vector fields}\label{secfphvf}

%%%%%%%%%%%%%%%%%%%%
% PROPOSITION
%%%%%%%%%%%%%%%%%%%%
\begin{proposition}\label{propcommutators}
Let $\mu,\nu$ be holomorphic $1$-forms on $\mathcal{A}$. We have:
\begin{eqnarray}
\left[ X_\mu , X_\nu \right] &=& 0, \label{equcomm1} \\
\left[\xi , X_\mu\right] &=& X_\tau, \label{equcomm2}
\end{eqnarray}
where $\tau = \mathcal{L}_{\xi^\mathcal{A}}(\mu) - \mu$.
\end{proposition}
\begin{proof}
Starting with the identity $\mathcal{L}_{X_\mu} \left( i_{X_\nu} \Omega_I \right) = i_{[X_\mu , X_\nu]} \Omega_I + i_{X_\nu} \mathcal{L}_{X_\mu} \Omega_I$, we have
\begin{equation*}
\begin{aligned}
i_{[X_\mu , X_\nu]} \Omega_I &= \mathcal{L}_{X_\mu} \left( i_{X_\nu} \Omega_I \right) - i_{X_\nu} \mathcal{L}_{X_\mu} \Omega_I \\
&= \mathcal{L}_{X_\mu} \left( h^*\nu \right) - i_{X_\nu} \left( d h^*\mu \right) \\
&= i_{X_\mu} \left( h^*(d\nu) \right) - i_{X_\nu} \left( h^*( d\mu ) \right ) \\
&= 0,
\end{aligned}
\end{equation*}
where we have used $h_* X_\mu = h_* X_\nu = 0$. This proves (\ref{equcomm1}).\\

Observe that $m_\lambda^* \Omega_I = \lambda \Omega_I$, for any $\lambda \in \mathbb{C}^*$. Thus $\mathcal{L}_\xi \Omega_I = \Omega_I$. Consider now the following computation:
\begin{equation*}
\begin{aligned}
i_{[\xi,X_\mu]} \Omega_I &= \mathcal{L}_\xi ( i_{X_\mu} \Omega_I ) - i_{X_\mu} \mathcal{L}_\xi \Omega_I \\
&= \mathcal{L}_\xi ( h^*(\mu) ) - i_{X_\mu} \Omega_I \\
&= h^*( \mathcal{L}_{\xi^\mathcal{A}}(\mu) ) - h^*(\mu) \\
&= h^*( \mathcal{L}_{\xi^\mathcal{A}} (\mu) - \mu  ).
\end{aligned}
\end{equation*}
This proves (\ref{equcomm2}), where $\tau = \mathcal{L}_{\xi^\mathcal{A}} (\mu) - \mu$.
\end{proof}
%%%%%%%%%%%%%%%%%%%%%%

%%%%%%%%%%%%%%%%%%%%%%%%%
% COROLLARY
%%%%%%%%%%%%%%%%%%%%%%%%%
\begin{corollary}\label{corapush}
Let $\mu$ be a holomorphic $1$-form on $\mathcal{A}$. Then
\begin{equation*}
(e^{X_\mu})_* \xi = \xi + X_\tau,
\end{equation*}
where $\tau = \mathcal{L}_{\xi^\mathcal{A}} (\mu) - \mu$.
\end{corollary}
\begin{proof}
Let $\xi_t = (e^{tX_\mu})_* \xi$. Then $\xi_0 = \xi$ and $\tfrac{d\xi_t}{dt} = [\xi , X_\mu] = X_\tau$, by Proposition \ref{propcommutators}. The result follows by integration.
\end{proof}
%%%%%%%%%%%%%%%%%%%%%%%%%%

%%%%%%%%%%%%%%%%%%%%%%%%%%%%%%%
% COROLLARY %%%%
%%%%%%%%%%%%%%%%%%%%%
\begin{corollary}\label{corinjects}
Let $\mu$ be a holomorphic $1$-form on $\mathcal{A}$. The automorphism $e^{X_\mu} : \mathcal{M}_{n,d} \to \mathcal{M}_{n,d}$ commutes with the $\mathbb{C}^*$-action of and only if $\mu = 0$. In particular $e^{X_\mu}$ is the identity if and only if $\mu = 0$.
\end{corollary}
\begin{proof}
As the $\mathbb{C}^*$-action is generated by $\xi$, we have that $e^{X_\mu}$ commutes with this action if and only if $(e^{X_\mu})_* \xi = \xi$. By Corollary \ref{corapush}, this happens if and only if $\mathcal{L}_{\xi^\mathcal{A}}(\mu) = \mu$, that is, $\mu$ is invariant under the $\mathbb{C}^*$-action on $\mathcal{A}$. It is easy to see that the only invariant $1$-form is $\mu = 0$. The second part of the Corollary follows, since the identity commutes with the $\mathbb{C}^*$-action.
\end{proof}
%%%%%%%%%%%%%%%%%%%%%%%%

Let $H^0(\mathcal{A} , \Omega^1(\mathcal{A})   )$ be the space of holomorphic $1$-forms on $\mathcal{A}$, which is an abelian group under addition. Consider the map $e^{X_{(.)} } : H^0(\mathcal{A} , \Omega^1(\mathcal{A}) ) \to Aut( \mathcal{M}_{n,d})$ sending $\mu$ to $e^{X_\mu}$. Equation (\ref{equcomm1}) implies that $e^{X_\mu} e^{X_\nu} = e^{X_\mu + X_\nu} = e^{X_{\mu + \nu} }$, so this map is a group homomorphism. Moreover, Corollary \ref{corinjects} implies that the map is injective. Thus we have identified $H^0(\mathcal{A},\Omega^1(\mathcal{A}))$ as a subgroup of $Aut( \mathcal{M}_{n,d})$. We will denote the image of $H^0(\mathcal{A} , \Omega^1(\mathcal{A}))$ in $Aut( \mathcal{M}_{n,d})$ by $Vert_0(\mathcal{M}_{n,d})$ and call it the group of vertical translations of $\mathcal{M}_{n,d}$ connected to the identity. Indeed, on any non-singular fibre $h^{-1}(b)$ of $\mathcal{M}_{n,d}$, an element $e^{X_\mu}$ of this group acts as a translation in $h^{-1}(b)$.

%%%%%%%%%%%%%%%%%%%%%%%%%%%%%%%%%%%%%%%%%%%%%%%%%%%%%%%%%%%%%%
%%%%%%%%%%%%%%%%%%%%%%%%%%%%%%%%%%%%%%%%%%%%%%%%%%%%%%%%%%%%%%
\section{Proof of the main theorem}\label{secpmt}

%%%%%%%%%%%%%%%%%%%%%
% LEMMA %%%%%%%%%%%%%%%%%
%%%%%%%%%%%%%%%%%%%%%%%
\begin{lemma}\label{lemsolve}
For every holomorphic $1$-form $\mu$ on $\mathcal{A}$ there is a unique holomorphic $1$-form $\nu$ on $\mathcal{A}$ satisfying
\begin{equation*}
\mathcal{L}_{\xi^\mathcal{A}}(\nu) - \nu = \mu.
\end{equation*}
\end{lemma}
\begin{proof}
Let $z^1 , \dots , z^r$ be linear coordinates on $\mathcal{A}$ such that $z^i$ has weight $m_i$ with respect to the $\mathbb{C}^*$-action. This means that $\xi^\mathcal{A} = m_1z^1 \frac{\partial}{\partial z^1} + \dots + m_r z^r \frac{\partial}{\partial z^r}$. Moreover, since the subspace $H^0(\Sigma , K^j) \subset \mathcal{A}$ has weight $j$, we see that $m_i \ge 2$ for all $i$. Let $\nu$ be a holomorphic $1$-form on $\mathcal{A}$. Then $\nu = \nu_1(z) dz^1 + \dots + \nu_r(z) dz^r$, where $\nu_1, \dots , \nu_r$ are holomorphic functions on $\mathcal{A}$. Let $\mu = \mathcal{L}_{\xi^\mathcal{A}}(\nu) - \nu$. Then $\mu = \mu_1(z) dz^1 + \dots + \mu_r(z) dz^r$, where
\begin{equation}\label{equcondition}
\mu_i(z) = \xi^\mathcal{A}( \nu_i(z) ) + (m_i-1)\nu_i(z).
\end{equation}
Given $\mu_i(z)$, we wish to find a solution $\nu_i(z)$ to (\ref{equcondition}). For each $i$, consider the function $\nu_i(z)$ defined by
\begin{equation*}
\nu_i(z) = \left( \frac{1}{2\pi i } \right)^r \! \! \int_{|w^1|=\zeta^1} \! \! \! \! \! \! \dots \! \int_{|w^r| = \zeta^r} \! \! \mu_i(w) \! \left( \sum_{I} \frac{1}{( \sum_j m_j i_j + m_i -1)} \left( \frac{z}{w} \right)^I \right) \! \frac{dw^1}{w^1} \dots \frac{dw^r}{w^r},
\end{equation*}
where the sum $\sum_I$ is over multi-indices $I = (i_1 , i_2 , \dots , i_r)$ and for a given $z \in \mathcal{A}$, $\zeta_1 , \dots , \zeta_r$ are chosen large enough that $\sum_I \left( \tfrac{z}{w} \right)^I$ converges absolutely, e.g., $\zeta^i > |z^i|$ suffices. Note crucially that the denominator $(\sum_j m_j i_j + m_i -1)$ is always $\ge 1$ because $m_i \ge 2$ for all $i$. Clearly $\nu_i(z)$ is a globally defined holomorphic function. It is easy to check that $\nu_i(z)$ satisfies Equation (\ref{equcondition}). Thus $\nu = \nu_1(z)dz^1 + \dots + \nu_r(z)dz^r$ is a solution to $\mathcal{L}_{\xi^\mathcal{A}}(\nu) - \nu = \mu$. Uniqueness follows for if $\nu$ is a holomorphic $1$-form with $\mathcal{L}_{\xi^\mathcal{A}}(\nu) - \nu = 0$, then $\nu$ is invariant under the $\mathbb{C}^*$-action and as in the proof of Corollary \ref{corinjects}, this implies $\nu = 0$.
\end{proof}
%%%%%%%%%%%%%%%%%%%%%%

%%%%%%%%%%%%%%%%%%
% PROPOSITION %%%%%%%%
%%%%%%%%%%%%%%%%%%
\begin{proposition}\label{propfactor}
Let $\phi : \mathcal{M}_{n,d} \to \mathcal{M}_{n,d}$. There is a unique holomorphic $1$-form $\nu$ on $\mathcal{A}$ such that the composition $e^{X_\nu} \circ \phi $ commutes with the $\mathbb{C}^*$-action.
\end{proposition}
\begin{proof}
Let $U \subset \mathcal{M}_{n,d}^{\rm sm}$ be the points of $\mathcal{M}_{n,d}^{\rm sm}$ where $h_*$ is surjective. By Corollary \ref{corcodim}, the complement of $U$ in $\mathcal{M}_{n,d}^{\rm sm}$ has codimension $\ge 2$. For any $m \in \mathcal{M}_{n,d}^{\rm sm}$ let $A_m \subset T_m^* \mathcal{M}_{n,d}^{\rm sm}$ be the subspace spanned by differentials $dg(m)$, where $g$ is a holomorhpic function on $\mathcal{M}_{n,d}$. Since all holomorphic functions on $\mathcal{M}_{n,d}$ are pullbacks from $\mathcal{A}$ it is easy to see that $U$ is precisely the set of $m \in \mathcal{M}_{n,d}^{\rm sm}$ such that $dim(A_m) = \tfrac{1}{2} dim( \mathcal{M}_{n,d})$. Now if $g$ is any holomorphic function on $\mathcal{M}_{n,d}$ then $\phi^*g$ is also holomorphic. It follows that $\phi$ preserves $U$ and sends $A_m$ to $A_{\phi^{-1}(m)}$ under pullback of $1$-forms. But for any $m \in U$, we have that $A_m \subset T_m^* \mathcal{M}_{n,d}^{\rm sm}$ is the annihilator of $Ker(h_*) \subset T_m \mathcal{M}_{n,d}^{\rm sm}$. Thus for any $m \in U$, $\phi$ sends $Ker(h_*)_m$ to $Ker(h_*)_{\phi(m)}$ isomorphically.\\

Let $X$ be the holomorphic vector field on $\mathcal{M}^{\rm sm}_{n,d}$ given by $X = \phi_*\xi$. By Theorem \ref{theoremvecreg}, we may write $X$ in the form $X = h^*(f)\xi + X_\mu$, for a holomorphic function $f$ and holomorphic $1$-form $\mu$ on $\mathcal{A}$. We claim that $f$ is non-vanishing. It suffices to show that $h^*(f)$ is non-vanishing on $U$, for if $h^*(f)$ has a zero then it vanishes on a codimension $1$ subspace, which must therefore meet $U$. If $(h^*f)( \phi(m) ) = 0$, where $m \in U$ then $\phi_*(\xi_m) = X_{\phi(m)} \in Ker(h_*)_{\phi(m)}$. But this is impossible, as $\xi_m$ is not in $Ker(h_*)_m$. Thus $f$ is non-vanishing.\\

Let $\nu$ be the unique solution to $\mathcal{L}_{\xi^\mathcal{A}}(\nu) - \nu = -\mu/f$ guaranteed by Lemma \ref{lemsolve}. We then have:
\begin{equation*}
\begin{aligned}
\left( e^{X_\nu} \circ \phi \right)_* \xi &= \left( e^{X_\nu} \right)_* \phi_* \xi \\
&= \left( e^{X_\nu} \right)_* (h^*(f)\xi + X_\mu) \\
&= \left( e^{X_\nu} \right)_* (h^*(f)\xi) + X_\mu \\
&= h^*(f) \left( e^{X_\nu} \right)_* \xi + X_\mu \\
&= h^*(f) (\xi + X_{-\mu/f}) + X_\mu \\
&= h^*(f)\xi.
\end{aligned}
\end{equation*}
In this computation we have used the fact that $\left( e^{X_\nu} \right)_* X_\mu = X_\mu$, which follows from Equation (\ref{equcomm1}) and $\left(e^{X_\nu} \right)_*( h^*(f) \xi) = h^*(f) \left(e^{X_\nu} \right)_*(\xi)$ which follows from the fact that $h_*(X_\nu) = 0$. What we have shown is that $\psi = e^{X_\nu} \circ \phi$ sends $\mathbb{C}^*$-orbits to $\mathbb{C}^*$-orbits. As in the proof of Proposition \ref{propintegrate}, we deduce that if $p \in \mathcal{M}_{n,d}$ is a point where $\mathbb{C}^*$ acts freely then there is a $c \in \mathbb{C}^*$ such that for every $\lambda \in \mathbb{C}^*$ we have either $\psi( m_\lambda(p) ) = m_{c\lambda}( \psi(p))$, or $ \psi( m_\lambda(p) ) = m_{c\lambda^{-1}}( \psi(p))$. Putting $\lambda = 1$, we see that we must have $c = 1$. Differentiating we find that $\psi_*(\xi_p) = \pm \xi_{\psi(p)}$. Since the orbits where $\mathbb{C}^*$ acts freely are dense we conclude that $\psi_*(\xi) = \pm \xi$, hence $f = \pm 1$.\\

Suppose that $\psi_*(\xi) = -\xi$. Then $\psi : \mathcal{M}_{n,d} \to \mathcal{M}_{n,d}$ is an automorphism with the property that for all $\lambda \in \mathbb{C}^*$:
\begin{equation}\label{equanti}
\psi \circ m_\lambda = m_{\lambda^{-1}} \circ \psi.
\end{equation}
Let $p \in \mathcal{M}_{n,d}$ be any point with $h(p) \neq 0$. Then, since the Hitchin map is proper, the sequence $m_{1/k}( \psi(p)), k = 1,2, \dots$ has a convergent subsequence. On the other hand the sequence $m_k(p)$ does not have a convergent subsequence, so neither does the sequence $\psi( m_k(p) )$, since $\psi$ is a homeomorphism. This contradicts (\ref{equanti}), so we must have $\psi_*(\xi) = \xi$. This means that $\psi = e^{X_\nu} \circ \phi$ commutes with the $\mathbb{C}^*$-action. Uniqueness of $\nu$ follows from Corollary \ref{corinjects}.
\end{proof}
%%%%%%%%%%%%%%%%%%%

%%%%%%%%%%%%%%%%%%%%%%%%%%%%
% PROPOSITION %%%%%%%%%%%%%%%%
%%%%%%%%%%%%%%%%%%%%%%%%%
\begin{proposition}\label{propsemid}
Let $Aut_{\mathbb{C}^*}(\mathcal{M}_{n,d})$ be the subgroup of $Aut(\mathcal{M}_{n,d})$ consisting of automorphisms of $\mathcal{M}_{n,d}$ that commute with the $\mathbb{C}^*$-action. We have that $Vert_0(\mathcal{M}_{n,d})$ is a normal subgroup of $Aut(\mathcal{M}_{n,d})$ and that $Aut(\mathcal{M}_{n,d})$ is the semi-direct product:
\begin{equation*}
Aut(\mathcal{M}_{n,d}) = Aut_{\mathbb{C}^*}(\mathcal{M}_{n,d}) \ltimes Vert_0(\mathcal{M}_{n,d}).
\end{equation*}
\end{proposition}
\begin{proof}
By Proposition \ref{propfactor}, we need only show that $Vert_0(\mathcal{M}_{n,d})$ is a normal subgroup of $Aut(\mathcal{M}_{n,d})$. In fact it is enough to show that if $\phi \in Aut_{\mathbb{C}^*}(\mathcal{M}_{n,d})$ and $e^{X_\mu} \in Vert_0(\mathcal{M}_{n,d})$, then $\phi \circ e^{X_\mu} \circ \phi^{-1} \in Vert_0(\mathcal{M}_{n,d})$. Consider the $1$-parameter subgroup $\phi \circ e^{tX_\mu} \circ \phi^{-1}$. Clearly this subgroup integrates the vector field $Y = \phi_*(X_\mu)$. Therefore it suffices to show that $Y$ is a vertical vector field, that is $h_*(Y) = 0$. Arguing as in the proof of Proposition \ref{propfactor}, we see that any automorphism $\phi$ sends vertical vector fields to vertical vector fields. In particular, $Y$ is vertical and $\phi \circ e^{X_\mu} \circ \phi^{-1} = e^{Y} \in Vert_0(\mathcal{M}_{n,d})$.
\end{proof}
%%%%%%%%%%%%%%%%%%%%%%%%%%%%%

%%%%%%%%%%%%%%%%%%%%%%%%%%
% PROPOSITION
%%%%%%%%%%%%%%%%%%%%%%%%%%
\begin{proposition}\label{propstableendo}
Let $\alpha$ be a holomorphic endomorphism of $T\mathcal{SU}_{n,d}^{\rm s}$, where $\mathcal{SU}_{n,d}^{\rm s} \subseteq \mathcal{SU}_{n,d}$ is the locus of stable bundles. Then $\alpha$ is a constant multiple of the identity.
\end{proposition}
\begin{proof}
Let $\phi_s$ be the $1$-parameter family of automorphisms of $T^*\mathcal{SU}_{n,d}^{\rm s}$, which acts fibrewise by $(e^{s\alpha^t})$. Here $\alpha^t$ is the endomorphism of the cotangent bundle induced by $\alpha$. We claim that there are automorphisms $\psi_s : \mathcal{A} \to \mathcal{A}$ such that we have a commutative diagram
\begin{equation*}
\xymatrix{
T^*\mathcal{SU}_{n,d}^{\rm s} \ar[r]^{\phi_s} \ar[d]^h & T^*\mathcal{SU}_{n,d}^{\rm s} \ar[d]^h \\
\mathcal{A} \ar[r]^{\psi_s} & \mathcal{A}
}
\end{equation*}
and moreover the $\psi_s$ commute with the $\mathbb{C}^*$-action on $\mathcal{A}$. Indeed this follows by an argument identical to the proof of \cite[Proposition 2.1]{kp}. We have that $\psi_s$ preserves the discriminant locus $\mathcal{D} \subset \mathcal{A}$, by \cite[Proposition 2.2]{kp}. It follows that $\phi_s$ sends $T^*\mathcal{SU}_{n,d}^{\rm s} \cap \mathcal{M}_{n,d}^{\rm reg}$ to itself. For any $b \in \mathcal{A}^{\rm reg}$ we have that $\phi_s$ gives a birational isomorphism between $h^{-1}(b)$ and $h^{-1}( \psi_s(b) )$. This is a birational isomorphism of abelian varieties and it follows that it extends to an isomorphism between $h^{-1}(b)$ and $h^{-1}(\psi_s(b))$. Thus $\phi_s$ extends as a $1$-parameter family of automorphisms of $\mathcal{M}_{n,d}^{\rm reg}$ (c.f., \cite[Page 248]{kp}). By Theorem \ref{theoremvecreg}, we have that the vector field on $\mathcal{M}_{n,d}^{\rm reg}$ generating the $1$-parameter family $\psi_s$ has the form $X = h^*(f) \xi + X_\mu$, where $f$ is a holomorphic function on $\mathcal{A}^{\rm reg}$ and $\mu$ is a holomorphic $1$-form on $\mathcal{A}^{\rm reg}$. Moreover, $\phi_s$ is defined on $T^* \mathcal{SU}_{n,d}^{\rm s}$ and the restriction of the Hitchin map $h|_{T^* \mathcal{SU}_{n,d}^{\rm s}} : T^* \mathcal{SU}_{n,d}^{\rm s} \to \mathcal{A}$ is surjective \cite[Lemma 1.4]{kp}. It follows that $f$ and $\mu$ extend to the whole of $\mathcal{A}$. Then since $\phi_s$ commutes with the $\mathbb{C}^*$-action we can use Corollary \ref{corinjects} to deduce that $\mu = 0$ and $f$ is constant. Thus, since $X$ is a constant multiple of $\xi$ we see that the automorphisms $\phi_s$ are given by the $\mathbb{C}^*$-action and hence $\alpha$ is a multiple of the identity.
\end{proof}
%%%%%%%%%%%%%%%%%%%%%%%%

Let $\mathcal{SU}_{n,d}$ be the moduli space of rank $n$, degree $d$, semi-stable bundles with fixed determinant $L_0$. Any automorphism $\phi : \mathcal{SU}_{n,d} \to \mathcal{SU}_{n,d}$ can be differentiated giving an automorphism $\phi_* = (\phi_*)^{-1} : T^* \mathcal{SU}_{n,d}^{\rm sm} \to T^* \mathcal{SU}_{n,d}^{\rm sm}$. It is clear from Theorem \ref{theoremstableauto} that such automorphisms automatically extend to automorphisms of the Higgs bundle moduli space. Let $Aut(\mathcal{SU}_{n,d})$ be the group of automorphisms of $\mathcal{SU}_{n,d}$. We have just argued that $Aut(\mathcal{SU}_{n,d})$ is in a natural way a subgroup of $Aut(\mathcal{M}_{n,d})$.

%%%%%%%%%%%%%%%%%%%%%%%%%%%%%%%%%
% THEOREM %%%%%%%%%%%%%%%
%%%%%%%%%%%%%%%%%%%%%%%%%%%%%%%
\begin{theorem}\label{theoremautgroup}
Let $\Sigma$ have genus $g \ge 3$. We have an isomorphism $Aut_{\mathbb{C}^*}(\mathcal{M}_{n,d}) = \mathbb{C}^* \times Aut(\mathcal{SU}_{n,d})$, where the subgroup $\mathbb{C}^* \subset Aut_{\mathbb{C}^*}(\mathcal{M}_{n,d})$ is the usual $\mathbb{C}^*$-action on $\mathcal{M}$. Therefore, using Proposition \ref{propsemid}, we have an isomorphism:
\begin{equation*}
Aut(\mathcal{M}_{n,d}) = \left( \mathbb{C}^* \times Aut(\mathcal{SU}_{n,d}) \right) \ltimes Vert_0( \mathcal{M}_{n,d} ).
\end{equation*}
\end{theorem}
\begin{proof}
It is clear that $\mathbb{C}^* \times Aut(\mathcal{SU}_{n,d}) \subseteq Aut_{\mathbb{C}^*}(\mathcal{M}_{n,d})$, so we only need to show the reverse inclusion $Aut_{\mathbb{C}^*}(\mathcal{M}_{n,d}) \subseteq \mathbb{C}^* \times Aut(\mathcal{SU}_{n,d})$. Let $\phi : \mathcal{M}_{n,d} \to \mathcal{M}_{n,d}$ be an automorphism of $\mathcal{M}_{n,d}$ which commutes with the $\mathbb{C}^*$-action. Let $U \subset \mathcal{M}_{n,d}$ be the open subset $U = T^* \mathcal{SU}_{n,d}^{\rm s} \cap \phi(T^* \mathcal{SU}_{n,d}^{\rm s} )$. Note that the complement $V = \mathcal{M}_{n,d} - T^*\mathcal{SU}_{n,d}^{\rm s}$ is an analytic subset and that $U = \mathcal{M}_{n,d} - (V \cup \phi(V) )$, so $U$ is dense in $\mathcal{M}_{n,d}$. Let $W \subseteq \mathcal{SU}_{n,d}^{\rm s}$ be the image of $U$ under the projection $p : T^* \mathcal{SU}_{n,d}^{\rm s} \to \mathcal{SU}_{n,d}^{\rm s}$. Then $W$ is an open subset since $p$ is an open mapping and it is easy to see that $W$ is dense in $\mathcal{SU}^{\rm s}_{n,d}$, since $U \subseteq T^*\mathcal{SU}^{\rm s}_{n,d}$ is dense in $T^* \mathcal{SU}_{n,d}^{\rm s}$.\\

Let $E \in W$. By definition of $W$ this means that there is a Higgs field $\Phi \in H^0( \Sigma , End_0(E) \otimes K )$, such that $(E,\Phi) \in U$. In turn, this means that $E$ is a stable bundle and $\phi(E,\Phi) = (F , \Phi')$, where $F$ is also stable. By $\mathbb{C}^*$-invariance it follows that $\phi( E , \lambda \Phi ) = (F , \lambda \Phi')$ for all $\lambda \in \mathbb{C}^*$. Taking the limit as $\lambda \to 0$ and using continuity of $\phi$, we get $\phi( E , 0 ) = ( F , 0 )$. If we think of $W$ as a subset of $\mathcal{M}_{n,d}$ by the inclusions $W \subseteq \mathcal{SU}_{n,d} \subset \mathcal{M}_{n,d}$, then we have just shown that $\phi(W) \subseteq \mathcal{SU}_{n,d}$. Then since $\mathcal{SU}_{n,d}$ is closed in $\mathcal{M}_{n,d}$ and since $W$ is dense in $\mathcal{SU}_{n,d}$, we have $\phi( \mathcal{SU}_{n,d} ) \subseteq \mathcal{SU}_{n,d}$ by continuity. This shows that $\phi$ restricts to an automorphism of $\mathcal{SU}_{n,d}$, i.e., there exists $\psi \in Aut( \mathcal{SU}_{n,d} )$ such that $\phi|_{\mathcal{SU}_{n,d}} = \psi$.\\

Let $(E , \Phi ) \in T^* \mathcal{SU}_{n,d}^{\rm s}$ and set $(F , \Phi' ) = \phi(E,\Phi)$. So by $\mathbb{C}^*$-equivariance, $\phi( E , \lambda \Phi ) = ( F , \lambda \Phi')$. In the limit as $\lambda \to 0$, we have by continuity of $\phi$ that $(F , \lambda \Phi' ) \to \phi(E,0) = (\psi(E) , 0 ) \in \mathcal{SU}_{n,d}^{\rm s}$. Then since $T^* \mathcal{SU}_{n,d}^{\rm s}$ is open in $\mathcal{M}_{n,d}$ we have that $( F , \lambda \Phi' ) \in T^* \mathcal{SU}_{n,d}^{\rm s}$, for all small enough $\lambda$. Thus $F$ is stable and $(F , \lambda \Phi')$ is in $T^*  \mathcal{SU}_{n,d}^{\rm s}$ for all $\lambda$. In particular setting $\lambda = 1$, we get that $(F , \Phi') = \phi(E , \Phi) \in T^* \mathcal{SU}_{n,d}^{\rm s}$. This shows that $\phi$ restricts to an automorphism of $T^* \mathcal{SU}_{n,d}^{\rm s}$. Our argument also shows that $p(\phi(m)) = \psi( p(m) )$ for any $m \in T^* \mathcal{SU}_{n,d}^{\rm s}$.\\

Let $\psi_* = (\psi^*)^{-1} : T^* \mathcal{SU}_{n,d}^{\rm s} \to T^* \mathcal{SU}_{n,d}^{\rm s}$ be the automorphism of $T^* \mathcal{SU}_{n,d}^{\rm s}$ obtained by differentiating $\psi$. From Theorem \ref{theoremstableauto}, we see that $\psi_*$ extends to an automorphism of $\mathcal{M}_{n,d}$ which commutes with the $\mathbb{C}^*$-action. Composing $\phi$ with $(\psi_*)^{-1}$, we reduce to the case that $\phi |_{\mathcal{SU}_{n,d}} = id$. So the restriction of $\phi$ to $T^* \mathcal{SU}_{n,d}^{\rm s}$ acts as a fibre preserving automorphism covering the identity. Since $\phi$ commutes with the $\mathbb{C}^*$-action, $\phi$ descends to an automorphism of the projective cotangent bundle of $\mathcal{SU}_{n,d}^{\rm s}$. This shows that $\phi$ acts linearly on the fibres of $T^* \mathcal{SU}_{n,d}^{\rm s}$. It follows from Proposition \ref{propstableendo}, that such an automorphism is given by the $\mathbb{C}^*$-action and the theorem follows.
\end{proof}
%%%%%%%%%%%%%%%%%%%%%%%%%%%%%%%

%%%%%%%%%%%%%%%%%%%%%%%%%%%%%%%%%%%%%%%%%%%%%%%%%%%%%%%%%%%%%%%%%%%
%%%%%%%%%%%%%%%%%%%%%%%%%%%%%%%%%%%%%%%%%%%%%%%%%%%%%%%%%%%%%%%%%%%
\section{Subgroups preserving additional structures}\label{secstructures}

%%%%%%%%%%%%%%%%%%%%%%%%%%%%%%%%%%%%%%

\subsection{Hyper-K\"ahler geometry of the Higgs bundle moduli space}\label{sechyperkg}

As we recall, the moduli space $\mathcal{M}_{n,d}$ carries a natural hyper-K\"ahler structure. To describe this we need to recall that $\mathcal{M}_{n,d}$ can also be viewed as the moduli space of solutions to the Hitchin equations. Let $E$ be a fixed choice of a smooth, rank $n$ degree $d$ complex vector bundle, equip $E$ with a Hermitian metric and let $L_0 = det(E)$ with the induced metric. We let $\mathfrak{sl}(E) = End_0(E)$ be the bundle of trace-free endomorphisms of $E$ and $\mathfrak{su}(E) \subset \mathfrak{sl}(E)$ the bundle of skew-adjoint trace-free endomorphisms of $E$. We let $\Omega^j(\Sigma , \mathfrak{sl}(E))$ denote the space of $j$-form valued sections of $\mathfrak{sl}(E)$. The adjoint map $A \mapsto A^*$ extends to an anti-linear involution $( \, . \, )^* : \Omega^j(\Sigma , \mathfrak{sl}(E) ) \to \Omega^j(\Sigma , \mathfrak{sl}(E) )$ and the Hodge star $\ast$ extends to a linear map $\ast : \Omega^j(\Sigma , \mathfrak{sl}(E) ) \to \Omega^{2-j}( \Sigma , \mathfrak{sl}(E) )$.\\

The complex structure on $\Sigma$ gives $\Sigma$ an orientation. Let $vol_\Sigma$ be a volume form on $\Sigma$ inducing the same orientation and such that $\int_\Sigma vol_\Sigma = 1$. This determines a real valued $L^2$-inner product on $\Omega^*( \Sigma , \mathfrak{sl}(E) )$:
\begin{equation*}
\langle \alpha , \beta \rangle = \frac{1}{2} \int_\Sigma Tr( \alpha^* \wedge \ast \beta ) + Tr( \beta^* \wedge \ast \alpha ).
\end{equation*}
If $\overline{\partial}_E$ is a $\overline{\partial}$-operator on $E$, we let $\nabla_E$ denote the associated Chern connection, the unique unitary connection on $E$ such that $(\nabla_E)^{0,1} = \overline{\partial}_E$ and we let $F_E \in \Omega^2 ( \Sigma , \mathfrak{su}_E )$ be the curvature of $\nabla_E$. Fix a choice of hermitian connection $\nabla_{L_0}$ on $L_0$ with curvature $F_{L_0} = -2\pi i d \, vol_\Sigma$. If $\overline{\partial}_E$ is a holomorphic structure on $E$, we will say that $det(E , \nabla_E) = (L_0,\nabla_{L_0})$ if the connection on $L_0$ induced by $\nabla_E$ equals $\nabla_{L_0}$.\\

Let $(\overline{\partial}_E , \Phi )$ be a pair consisting of a $\overline{\partial}$-operator $\overline{\partial}_E$ on $E$ such that $det(E , \nabla_E) \cong (L_0,\nabla_{L_0})$ and $\Phi$ a section of $\Omega^{1,0}(\Sigma , \mathfrak{sl}(E) )$. The space of such pairs is an affine space modelled on $\Omega^{0,1}(\Sigma , \mathfrak{sl}(E)) \oplus \Omega^{1,0}(\Sigma , \mathfrak{sl}(E))$. The {\em Hitchin equations} for $(\overline{\partial}_E , \Phi )$ are:
\begin{equation*}
\begin{aligned}
F_E + [ \Phi , \Phi^*] &= -2\pi i\mu(E) vol_\Sigma \otimes Id ,\\
\overline{\partial}_E \Phi &= 0.
\end{aligned}
\end{equation*}

Let $\mathcal{M}^{\rm Hit}_{n,L_0,\nabla_{L_0}}$ denote the moduli space of solutions to the Hitchin equations modulo unitary gauge transformations (of rank $n$, with trace-free Higgs field and fixed determinant $(L_0 , \nabla_{L_0}))$. Standard gauge-theoretic constructions give a topology on $\mathcal{M}^{\rm Hit}_{n,L_0,\nabla_{L_0}}$. Observe that if $(E,\Phi)$ is a solution to the Hitchin equations, then $(E,\Phi)$ is a Higgs bundle. Moreover, it can be shown that $(E,\Phi)$ is {\em polystable}, i.e., a direct sum of stable Higgs bundles of the same slope. Since polystable Higgs bundles are semi-stable we have a natural map $\iota : \mathcal{M}^{\rm Hit}_{n,L_0,\nabla_{L_0}} \to \mathcal{M}_{n,L_0}$. A theorem of Hitchin \cite{hit1} and Simpson \cite{sim1} establishes a Hitchin-Kobayashi type correspondence for Higgs bundles. This correspondence states that a Higgs bundle $(E,\Phi)$ is in the image of $\iota$ if and only if it is polystable. But every $S$-equivalence class of Higgs bundles has a unique polystable object, so $\iota$ is a bijection, in fact a homeomorphism.

\begin{remark}
The moduli space $\mathcal{M}^{Hit}_{n,L_0,\nabla_{L_0}}$ essentially depends on $(L_0 , \nabla_{L_0})$ only through the degree $d$ mod $n$. To see this, let $(L , \nabla_L)$ be a line bundle of degree $a$ and let $\nabla_L$ be a connection on $L$ with curvature $F_L = -2\pi i a \, vol_\Sigma$. Tensoring solutions of the Hitchin equations by $(L, \nabla_L)$ produces a commutative square:
\begin{equation*}
\xymatrix{
\mathcal{M}^{\rm Hit}_{n,L_0,\nabla_{L_0}} \ar[d]^-{\iota} \ar[rr]^-{\otimes (L , \nabla_L ) } & & \mathcal{M}^{\rm Hit}_{n,L_0\otimes L^n , \nabla_{L_0} \otimes Id + Id \otimes (\nabla_L)^{\otimes n} } \ar[d]^-{\iota} \\
\mathcal{M}_{n,L_0} \ar[rr]^-{\otimes L } & & \mathcal{M}_{n,L_0 \otimes L^n}
}
\end{equation*}
In a similar manner, one can show that the choice of volume form $vol_\Sigma$ is completely arbitrary. As such we may safely write $\mathcal{M}^{\rm Hit}_{n,d}$ for the moduli space of solutions of the Hitchin equations and observe that we have a homeomorphism $\mathcal{M}^{\rm Hit}_{n,d} \cong \mathcal{M}_{n,d}$.
\end{remark}

One upshot of the identification $\mathcal{M}_{n,d} \cong \mathcal{M}_{n,d}^{\rm Hit}$ is that $\mathcal{M}_{n,d}^{\rm Hit}$ carries a natural hyper-K\"ahler structure. Indeed, $\mathcal{M}_{n,d}^{\rm Hit}$ may be constructed as an infinite dimensional hyper-K\"ahler quotient \cite{hit1}. Let $m = (\overline{\partial}_E , \Phi) \in \mathcal{M}_{n,d}^{\rm Hit}$ be a smooth point. From the hyper-K\"ahler quotient construction, it follows that the tangent space $T_m \mathcal{M}^{\rm Hit}_{n,d}$ can be described in terms of harmonic representatives. Under this identification the tangent space $T_m \mathcal{M}^{\rm Hit}_{n,d}$ is given by pairs $(\dot{A} , \dot{\Phi}) \in \Omega^{0,1}(\Sigma , \mathfrak{sl}(E)) \oplus \Omega^{1,0}(\Sigma , \mathfrak{sl}(E))$ such that:
\begin{equation*}
\begin{aligned}
\overline{\partial}_E \dot{\Phi} + [ \dot{A} , \Phi ] &= 0, \\
\partial_E \dot{A} + [\dot{\Phi} , \Phi^* ] &= 0,
\end{aligned}
\end{equation*}
where $\partial_E$ is the $(1,0)$-part of the Chern connection $\nabla_E$ associated to $\overline{\partial}_E$. The hyper-K\"ahler structure on the smooth locus of $\mathcal{M}^{\rm Hit}_{n,d}$ is given by a metric $g$ and complex structures $I,J,K$ satisfying the quaternionic relations $IJ = K$. In terms of harmonic representatives the metric $g$ is given by:
\begin{equation*}
g( (\dot{A}_1 , \dot{\Phi}_1) , (\dot{A}_2 , \dot{\Phi}_2) ) = \frac{i}{2} \int_\Sigma Tr( \dot{A}^*_1 \wedge \dot{A}_2 + \dot{A}^*_2 \wedge \dot{A}_1 ) - Tr( \dot{\Phi}^*_1 \wedge \dot{\Phi}_2 + \dot{\Phi}^*_2 \wedge \dot{\Phi}_1 ),
\end{equation*}
and the complex structures $I,J,K$ by:
\begin{equation*}
\begin{aligned}
I(\dot{A},\dot{\Phi}) = (i\dot{A} , i\dot{\Phi}), && J(\dot{A},\dot{\Phi}) = (i\dot{\Phi}^* , -i\dot{A}^*), && K(\dot{A},\dot{\Phi})  = (-\dot{\Phi}^* , \dot{A}^*).
\end{aligned}
\end{equation*}
Note that $I$ is just the natural complex structure on $\mathcal{M}_{n,d}$ as introduced in Section \ref{sechiggsbundles}. Let $\omega_I , \omega_J , \omega_K$ denote the associated K\"ahler forms:
\begin{equation*}
\begin{aligned}
\omega_I(X,Y) = g(IX,Y), && \omega_J(X,Y) = g(JX,Y), && \omega_K(KX,Y) = g(KX,Y).
\end{aligned}
\end{equation*}
We also define complex $2$-forms $\Omega_I,\Omega_J,\Omega_K$ by:
\begin{equation*}
\begin{aligned}
\Omega_I = \omega_J + i\omega_K, && \Omega_J = \omega_K + i\omega_I, && \Omega_K = \omega_I + i\omega_J.
\end{aligned}
\end{equation*}
Then $\Omega_I$ is a closed complex symplectic $2$-form of type $(2,0)$ with respect to $I$ and similarly for $\Omega_J,\Omega_K$. Note that this definition of $\Omega_I$ agrees with our previous definition, Equation (\ref{equomegai1}).

%%%%%%%%%%%%%%%%%%%%%%%%%%%%%%%%%%%%%%%%%%%%%

\subsection{Symmetry groups}\label{secsymgp}

Our goal in this section is to determine the subgroups given in Definition \ref{defautgroups0}.

%%%%%%%%%%%%%%%%%%%%%%%%
% LEMMA %%%%%%%%
%%%%%%%%%%%%%%%%%%%%%%%%
\begin{lemma}\label{lemsimplyconn}
Suppose $g \ge 3$. Then $\mathcal{M}_{n,d}^{\rm sm}$ is simply-connected.
\end{lemma}
\begin{proof}
Let $\mathcal{SU}_{n,d}^s \subset \mathcal{SU}_{n,d}$ be the locus of stable bundles in $\mathcal{SU}_{n,d}$. We have that $\mathcal{SU}_{n,d}^s$ is simply-connected \cite{bbgn,dn}. It is also known that for $g \ge 3$, $\mathcal{SU}_{n,d}^s = \mathcal{SU}_{n,d}^{\rm sm}$ \cite{nr}. In particular, it follows that $T^*\mathcal{SU}_{n,d}^{\rm sm}$ is simply-connected. The lemma follows since the codimension of the complement of $T^*\mathcal{SU}_{n,d}^{\rm sm} \subseteq \mathcal{M}_{n,d}^{\rm sm}$ is at least $2$.
\end{proof}
%%%%%%%%%%%%%%%%%%%%%%%%%%

%%%%%%%%%%%%%%%%%%%%%%%%
% LEMMA %%%%%%%%
%%%%%%%%%%%%%%%%%%%%%%%%
\begin{lemma}\label{lemkilling}
Let $\mu$ be a holomorphic $1$-form on $\mathcal{A}$ and $X_\mu$ the corresponding holomorphic vector field on $\mathcal{M}_{n,d}$.
\begin{enumerate}
\item{If $X_{\mu}$ preserves $J$ (i.e., $\mathcal{L}_{X_\mu}J = 0$), then $\mu = 0$.}
\item{If $X_{\mu}$ preserves $g$ (i.e., $\mathcal{L}_{X_\mu}g = 0$), then $\mu = 0$.}
\end{enumerate}
\end{lemma}
\begin{proof}
We first show that condition (1) implies condition (2), that is, if $X_\mu$ preserves $J$ then it also preserves $g$. To see this suppose that $X_\mu$ preserves $J$. Over $\mathcal{M}_{n,d}^{\rm reg}$ we have an orthogonal decomposition $T\mathcal{M}_{n,d}^{\rm reg} = {\rm Ker}(h_*) \oplus J \,{\rm Ker}(h_*)$. Let $V = {\rm Ker}(h_*)$ and $H = JV$. Since $X_\mu$ preserves the subbundle $V$, it must also preserve $H$. It follows that $\mathcal{L}_{X_\mu} g$ is a section of $S^2(V^*) \oplus S^2(H^*)$. On the other hand observe that $\Omega_I(J X , Y) = \omega_J(JX,Y) + i\omega_K(JX,Y) = -g(X,Y) -i\omega_I(X,Y)$. So for any real vector fields $X,Y$ we have $g(X,Y) = -Re( \Omega_I(JX,Y) )$. However we also see that
\begin{equation*}
\begin{aligned}
\mathcal{L}_{X_\mu} \Omega_I &= d i_{X_\mu} \Omega_I + i_{X_\mu} d \Omega_I \\
&= d h^*(\mu) = h^*( d\mu ).
\end{aligned}
\end{equation*}
So if $X_\mu$ preserves $J$, then for all $X,Y$ we have:
\begin{equation*}
\mathcal{L}_{X_\mu}g(X,Y)  = -Re( h^*(d\mu) ( JX , Y ) ).
\end{equation*}
However the right hand side vanishes whenever $X$ and $Y$ are either both vertical or both horizontal. So this equality is only possible if both sides vanish and hence $X_\mu$ preserves $g$.\\

Now suppose that $X_\mu$ preserves $g$. Clearly this implies that $X_\mu$ preserves $\omega_I$, so $ i_{X_\mu} \omega_I $ is a closed $1$-form on $\mathcal{M}_{n,d}^{\rm sm}$. By Lemma \ref{lemsimplyconn}, $\mathcal{M}_{n,d}^{\rm sm}$ is simply-connected so there is a smooth function $g$ on $\mathcal{M}_{n,d}^{\rm sm}$ such that $i_{X_\mu} \omega_I = dg$. On the other hand, given a non-singular fibre $h^{-1}(b) \subset \mathcal{M}_{n,d}^{\rm reg}$, we have that $\omega_I$ restricts to a K\"ahler form on $h^{-1}(b)$ and that the flow of $X_\mu$ on $h^{-1}(b)$ is given by translations. As is well known, the action of a complex torus on itself by translation is not Hamiltonian. Hence we have a contradiction unless $X_\mu$ vanishes on $h^{-1}(b)$. Since $h^{-1}(b)$ was an arbitrary non-singular fibre this shows that $X_\mu = 0$ and hence $\mu = 0$.
\end{proof}
%%%%%%%%%%%%%%%%%%%%%%%%

%%%%%%%%%%%%%%%%%%%%%%%%%%%
% COROLLARY
%%%%%%%%%%%%%%%%%%%%%%%%%%%%%
\begin{corollary}\label{corkilling}
Let $Y = f\xi + X_\mu$ be a holomorphic vector field on $\mathcal{M}_{n,d}^{\rm sm}$, where $f \in \mathbb{C}$ is a constant and $\mu$ is a holomorphic $1$-form on $\mathcal{A}$. 
\begin{enumerate}
\item{If $Y$ preserves $J$, then $\mu = 0$ and $f \in \mathbb{R}$.}
\item{If $Y$ preserves $g$, then $\mu = 0$ and $f \in i\mathbb{R}$.}
\end{enumerate}
\end{corollary}
\begin{proof}
(1). First we note that $\xi$ preserves $J$ but does not preserve $g$, while $i\xi$ preserves $g$ but does not preserve $J$. If $Y$ preserves $J$ then so does the commutator $[\xi , Y ] = [\xi , f\xi + X_\mu ] = X_\tau$, where $\tau = \mathcal{L}_{\xi^\mathcal{A}}(\mu) - \mu$. Then by Lemma \ref{lemkilling}, we have $\tau = 0$, hence $\mu = 0$. So $Y = f\xi$ and for this to preserve $J$ we must have that $f$ is real. The proof of (2) follows by a similar argument.
\end{proof}
%%%%%%%%%%%%%%%%%%%%%%%%%%%%%%%%%%

For any holomorphic function $f$ on $\mathcal{A}$, we have the corresponding Hamiltonian vector field $X_{f}$ which can be integrated to a Hamiltonian flow $e^{X_{f}}$. Clearly the Hamiltonian flows define a subgroup of $Vert_0(\mathcal{M}_{n,d})$, which we will denote by $Ham(\mathcal{M}_{n,d})$. Then the map $\mathcal{O}(\mathcal{A}) \to Ham(\mathcal{M}_{n,d})$ sending a holomorphic function $f$ to $e^{X_{f}}$ is a surjective homomorphism with Kernel the constant functions on $\mathcal{A}$.

%%%%%%%%%%%%%%%%%%%%%%%%%%%%%%%%
% PROPOSITION
%%%%%%%%%%%%%%%%%%%%%%%%%%%%%%%%%
\begin{theorem}
Under the isomorphism $Aut(\mathcal{M}_{n,d}) \cong \left( \mathbb{C}^* \times Aut(\mathcal{SU}_{n,d}) \right) \ltimes Vert_0(\mathcal{M}_{n,d})$ of Theorem \ref{theoremautgroup}, the subgroups given in Definition \ref{defautgroups0} are as follows:
\begin{enumerate}
\item{$Aut_{Sympl}(\mathcal{M}_{n,d}) = \left( \{1\} \times Aut( \mathcal{SU}_{n,d} ) \right) \ltimes Ham(\mathcal{M}_{n,d})$,}
\item{$Aut_{Isom}(\mathcal{M}_{n,d}) = \left( U(1) \times Aut( \mathcal{SU}_{n,d} ) \right)$,}
\item{$Aut_{Q}(\mathcal{M}_{n,d}) = \left( \mathbb{R}_{+} \times Aut( \mathcal{SU}_{n,d} ) \right)$,}
\item{$Aut_{HK}(\mathcal{M}_{n,d}) = \left( \{ 1 \} \times Aut( \mathcal{SU}_{n,d} ) \right)$,}
\end{enumerate}
where $U(1) \subset \mathbb{C}^*$ is the subgroup of unit complex numbers and $\mathbb{R}_+ \subset \mathbb{C}^*$ is the subgroup of positive real numbers.
\end{theorem}
\begin{proof}
From the description of $g,I,J,K$ given in Section \ref{sechyperkg}, it is straightforward to see that the subgroup $Aut(\mathcal{SU}_{n,d})$ preserves the full hyper-K\"ahler structure. Therefore it suffices to only consider elements in $\mathbb{C}^* \ltimes Vert_0( \mathcal{M}_{n,d} )$. Such an element can be written in the form $\phi = e^{X_\mu} \circ m_\lambda$, where $\mu$ is a holomorphic $1$-form on $\mathcal{A}$ and $\lambda \in \mathbb{C}^*$.\\

(1). For this note that $(\phi^{-1})^* \Omega_I = (e^{-X_\mu})^* m_{\lambda^{-1}}^* \Omega_I = (e^{-X_\mu})^* \lambda^{-1} \Omega_I = \lambda^{-1}( \Omega_I - h^*(d\mu))$. So $\phi$ preserves $\Omega_I$ if and only if $\Omega_I = \lambda^{-1}( \Omega_I - h^*(d\mu))$. Clearly this is possible if and only of $\lambda = 1$ and $d\mu = 0$. Therefore $\mu$ is a closed $1$-form on $\mathcal{A}$ and there exists a holomorphic function $f$ such that $\mu = df$.\\

(2). Since $i\xi$ preserves $g$, it follows that so does $(e^{X_\mu} \circ m_\lambda )_* i\xi = ( e^{X_\mu} )_* {m_\lambda}_* i\xi = ( e^{X_\mu} )_* (i\xi) = i(\xi + X_\tau)$, where $\tau = \mathcal{L}_{\xi^\mathcal{A}}(\mu) - \mu$. From Corollary \ref{corkilling}, it follows that $\tau = 0$ and therefore $\mu = 0$.\\

(3). Since $\xi$ preserves $J$, it follows that so does $(e^{X_\mu} \circ m_\lambda )_* \xi = ( e^{X_\mu} )_* {m_\lambda}_* \xi = ( e^{X_\mu} )_* (\xi) = (\xi + X_\tau)$, where $\tau = \mathcal{L}_{\xi^\mathcal{A}}(\mu) - \mu$. Once again, it follows from Corollary \ref{corkilling}, that $\mu = 0$.\\

(4). This follows immediately from cases (2) and (3).
\end{proof}
%%%%%%%%%%%%%%%%%%%%%%%%

%%%%%%%%%%%%%%%%%%%%%%%%%%%%%%%%%%%%%%%%%%%%%%%%%%%%%%%%%%%%%
%%%%%%%%%%%%%%%%%%%%%%%%%%%%%%%%%%%%%%%%%%%%%%%%%%%%%%%%%%%%%

\section{Anti-automorphisms and the full isometry group}\label{secisometry}

%%%%%%%%%%%%%%%%%%%%%%%%%%%%%%%%%%%%%%%%%%

\subsection{Anti-automorphisms}\label{secanti}

In this section we will make the simplifying assumption that $n$ and $d$ are coprime. It follows that all semi-stable Higgs bundles are stable and that $\mathcal{M}_{n,d}$ is a smooth hyper-K\"ahler manifold. Likewise all semi-stable bundles are stable and $\mathcal{SU}_{n,d}$ is a smooth projective variety. If $X$ is a complex manifold with complex structure $I$ then by an {\em anti-automorphism} of $X$ we mean a diffeomorphism $\phi : X \to X$ such that $\phi_* \circ I = -I \circ \phi_*$.

%%%%%%%%%%%%%%%%%%%%%%%%%%%
% PROPOSITION %%%%%%%%%%
%%%%%%%%%%%%%%%%%%%%%%%%%%%
\begin{theorem}\label{theoremantiauto}
Suppose that $n$ and $d$ are coprime. Then $\mathcal{M}_{n,d}$ admits an anti-automorphism if and only if $\Sigma$ admits an anti-automorphism.
\end{theorem}
\begin{proof}
First suppose that $\mathcal{M}_{n,L_0}$ admits an anti-automorphism $f : \Sigma \to \Sigma$. Fix an underlying smooth bundle $E$ of rank $n$, degree $d$ and choose a Hermitian structure on $E$. As in Section \ref{secstructures}, we identify $\mathcal{M}_{n,L_0}$ with $\mathcal{M}_{n,L_0}^{\rm Hit}$, the moduli space of solutions to the Hitchin equations. We can view $L_0$ as a line bundle with unitary connection. Then $f^*(L_0)$ is also a line bundle with unitary connection and in particular has a holomorphic structure. Bearing this in mind, we obtain an anti-holomorphic map $f^* : \mathcal{M}_{n,L_0} \to \mathcal{M}_{n,f^*(L_0)}$ which sends a solution $(\overline{\partial}_E , \Phi )$ of the Hitchin equations to $f^*(E,\Phi ) = ( f^*(\partial_E ) , f^*( \Phi^* ) )$, where $\partial_E$ is the $(1,0)$-part of the Chern connection of $\overline{\partial}_E$ (cf. \cite{bs}). This corresponds to pullback under $f$ of the connection $\nabla = \nabla_E + \Phi + \Phi^*$. Let $L$ be a line bundle with connection such that $L^n \cong L_0 f^*(L_0)$. Such an $L$ exists as $L_0 f^*(L_0)$ has degree zero. Next, consider the map $\delta_L : \mathcal{M}_{n,f^*(L_0)} \to \mathcal{M}_{n,f^*(L_0^*) L^n} = \mathcal{M}_{n,L_0}$, which is the holomorphic map $(E , \Phi ) \mapsto (E^* \otimes L , \Phi^t \otimes Id)$. The composition $\hat{f}_L = \delta_L \circ f^* : \mathcal{M}_{n,L_0} \to \mathcal{M}_{n,L_0}$ is then an anti-automorphism of $\mathcal{M}_{n,L_0}$.\\

To prove the converse, we introduce the following notation: fix a choice of volume form $vol_\Sigma$ on $\Sigma$. If $j$ is a complex structure on $\Sigma$ inducing the same orientation as $vol_\Sigma$, we let $\mathcal{M}_{n,L_0}(j)$ denote the moduli space of rank $n$, trace-free Higgs bundles with determinant $L_0$ associated to $(\Sigma , j)$. Now let $j$ be a given complex structure on $\Sigma$ and suppose that $\phi : \mathcal{M}_{n,L_0}(j) \to \mathcal{M}_{n,L_0}(j)$ is an anti-automorphism. Choose an orientation reversing homeomorphism $g : \Sigma \to \Sigma$ and suppose $vol_\Sigma$ is chosen such that $g^*(vol_\Sigma) = -vol_\Sigma$. Note that such pairs $(g , vol_\Sigma)$ certainly exist, as we could take $g$ to be an involution. Further, choose a unitary line bundle with connection $L$, such that $L^n \cong L_0 g^*(L_0)$. Then $-g^*(j)$ is a complex structure on $\Sigma$ inducing the same orientation as $j$. We have an anti-automorphism $g^* : \mathcal{M}_{n,L_0}(j) \to \mathcal{M}_{n,g^*(L_0)}(-g^*(j))$ given by $g^*( \overline{\partial}_E , \Phi ) = ( g^*(\partial_E) , g^*( \Phi^* ) )$. As above, consider the map $\delta_L : \mathcal{M}_{n,g^*(L_0)} \to \mathcal{M}_{n,g^*(L_0^*) L^n} = \mathcal{M}_{n,L_0}$, which is the holomorphic map $(E , \Phi ) \mapsto (E^* \otimes L , \Phi^t \otimes Id)$. The composition $\delta_L \circ g^* \circ \phi : \mathcal{M}_{n,L_0}(j) \to \mathcal{M}_{n,L_0}(-g^*(j) )$ is then an isomorphism of complex manifolds. Now we use the Torelli theorem for Higgs bundle moduli spaces \cite[Theorem 1.1]{bigo} to conclude that $(\Sigma , j)$ and $(\Sigma , -g^*(j))$ are isomorphic. This means that there is a diffeomorphism $f : \Sigma \to \Sigma $ such that $f^* (-g^*(j) ) = j$. Thus, $g \circ f : \Sigma \to \Sigma$ is an anti-automorphism of $(\Sigma , j)$.
\end{proof}
%%%%%%%%%%%%%%%%%%%%%%%%%%%%%%%

%%%%%%%%%%%%%%%%%%%%%%%%%%%%%%%%%%%%%%%%%%%%%%%

\subsection{The isometry group}\label{secisometry2}

In this section we will determine the isometry group $Isom(\mathcal{M}_{n,d})$ of $\mathcal{M}_{n,d}$.

%%%%%%%%%%%%%%%%%%%%%%%%%%%%%%%%
% LEMMA %%%%%%%%%
%%%%%%%%%%%%%%%%%%%%%%%%%%%%%
\begin{lemma}\label{lemconstant1}
Suppose $n$ and $d$ are coprime. The only covariantly constant endomorphisms of $T\mathcal{SU}_{n,d}$ are linear combinations of $I$ and the identity.
\end{lemma}
\begin{proof}
Any endomorphism $E :  T\mathcal{SU}_{n,d} \to T\mathcal{SU}_{n,d}$ can be uniquely written as a sum $E = A + B$, where $AI = IA$ and $BI = -IB$, namely $A = \tfrac{1}{2}(E - IEI)$, $B = \tfrac{1}{2}(E + IEI)$. If $E$ is covariantly constant then so are $A$ and $B$. In particular $A$ corresponds to a holomorphic endomorphism of $T^{(1,0)}\mathcal{SU}_{n,d}$. By Proposition \ref{propstableendo}, any such endomorphism is of the form $\lambda Id$, $\lambda \in \mathbb{C}$. This corresponds to $A = aId + b I$, where $\lambda = a + ib$, $a,b \in \mathbb{R}$. To finish the lemma it remains to show that there are no constant endomorphisms $B$ with $BI = -IB$. Such an endomorphism corresponds to an anti-linear map $B : T^{(1,0)}\mathcal{SU}_{n,d} \to T^{(0,1)}\mathcal{SU}_{n,d}$. The hermitian metric on $\mathcal{SU}_{n,d}$ defines an anti-linear isomorphism $h : T^{(0,1)}\mathcal{SU}_{n,d} \to (T^{(0,1)}\mathcal{SU}_{n,d})^*$. Thus the composition $h \circ B$ is a $\mathbb{C}$-linear covariantly constant endomorphism $h \circ B : T^{(1,0)}\mathcal{SU}_{n,d} \to (T^{(1,0)}\mathcal{SU}_{n,d})^*$.\\

If $h \circ B$ is an isomorphism then by taking determinants we obtain a trivialisation of the square of the canonical bundle. This contradicts the fact that the anti-canonical bundle is ample \cite{ram}. Therefore $h \circ B$ has a non-trivial kernel $U$, which has constant rank as $h \circ B$ is covariantly constant. Using $h$ we get an orthogonal decomposition $T^{(1,0)}\mathcal{SU}_{n,d} = U \oplus V$ which is preserved by the Levi-Civita connection. However if this decomposition is non-trivial we would obtain holomorphic endomorphisms of $T^{(1,0)}\mathcal{SU}_{n,d}$ other than multiples of the identity. Since this can not happen, $V = 0$, $U = T^{(1,0)}\mathcal{SU}_{n,d}$ and hence $B = 0$ as claimed.
\end{proof}
%%%%%%%%%%%%%%%%%%%%%%%%%%%%%%%%%

%%%%%%%%%%%%%%%%%%%%%%%%%%%%%%%%%
% PROPOSITION
%%%%%%%%%%%%%%%%%%%%%%%%%%%%%%%%%%%
\begin{proposition}\label{propconstantendo}
Suppose $n$ and $d$ are coprime. Then any covariantly constant endomorphism of $T\mathcal{M}_{n,d}$ is a linear combination of $I,J,K$ and the identity.
\end{proposition}
\begin{proof}
Consider the involution $\iota : \mathcal{M}_{n,d} \to \mathcal{M}_{n,d}$ given by $i(E,\Phi) = (E , -\Phi)$. This is an isometry of $\mathcal{M}_{n,d}$ and is anti-symplectic in the sense that $\iota^* \Omega_I = -\Omega_I$. It follows that the fixed point set of $\iota$ is a complex Lagrangian submanifold. If $\gamma(t)$ is a geodesic in $\mathcal{M}_{n,d}$ such that $\gamma'(0)$ is tangent to the fix point set, then the same is true of $\iota( \gamma(t) )$. Uniqueness of geodesics then gives $\gamma(t) = \iota(\gamma(t))$ and hence the fixed point set of $\iota$ is totally geodesic. Clearly $\mathcal{SU}_{n,d}$ is in the fixed point set of $\iota$ and since it is a closed submanifold of half the dimension on $\mathcal{M}_{n,d}$ it is actually a component of the fixed point set of $\iota$. This shows in particular that $\mathcal{SU}_{n,d} \subset \mathcal{M}_{n,d}$ is totally geodesic.\\

Let $\nabla_{\mathcal{M}}$ denote the Levi-Civita connection of $\mathcal{M}_{n,d}$ and $\nabla_{\mathcal{SU}}$ the Levi-Civita connection of $\mathcal{SU}_{n,d}$ with the induced metric. Since $\mathcal{SU}_{n,d} \subset \mathcal{M}_{n,d}$ is totally geodesic, we have that $\nabla_\mathcal{M}|_{\mathcal{SU}_{n,d}}$ respects the orthogonal decomposition $T\mathcal{M}_{n,d}|_{\mathcal{SU}_{n,d}} = T \mathcal{SU}_{n,d} \oplus N \mathcal{SU}_{n,d}$, where $N \mathcal{SU}_{n,d}$ is the normal bundle. We also have that the restriction of $\nabla_\mathcal{M}|_{\mathcal{SU}_{n,d}}$ to the sub-bundle $T \mathcal{SU}_{n,d}$ coincides with $\nabla_{\mathcal{SU}}$. Moreover, the complex structure $J$ gives a covariantly constant isomorphism $J : T \mathcal{SU}_{n,d} \to N \mathcal{SU}_{n,d}$. Using this isomorphism we have an isomorphism of bundles with connections:
\begin{equation}\label{equsplitting}
\left( T\mathcal{M}_{n,d}|_{\mathcal{SU}_{n,d}} ,  \mathcal{M}|_{\mathcal{SU}_{n,d}} \right) = \left( T \mathcal{SU}_{n,d} , \nabla_{\mathcal{SU}} \right) \oplus \left( T \mathcal{SU}_{n,d} , \nabla_{\mathcal{SU}} \right).
\end{equation}
Suppose that $\phi $ is a covariantly constant endomorphism of $T \mathcal{M}_{n,d}$. The restriction of $\phi$ to $\mathcal{SU}_{n,d}$ decomposes under (\ref{equsplitting}) into
\begin{equation}\label{equmatrix1}
\phi|_{\mathcal{SU}_{n,d}} = \left[ \begin{matrix} A & B \\ C & D \end{matrix} \right],
\end{equation}
where $A,B,C,D$ are covariantly constant endomorphisms of $T \mathcal{SU}_{n,d}$. By Lemma \ref{lemconstant1}, we have that $A,B,C,D$ are linear combinations of $I|_{\mathcal{SU}_{n,d}}$ and the identity. Let $\mathcal{C}$ be the space of all covariantly constant endomorphisms of $T \mathcal{M}_{n,d}$. We have just shown that $\mathcal{C}$ has real dimension at most $8$. On the other hand, as $I,J,K$ are covariantly constant, we have that $\mathcal{C}$ is a non-trivial module over the quaternions. So the dimension of $\mathcal{C}$ is either $4$ or $8$. If the dimension is $4$ we have proven the proposition, so suppose that $\mathcal{C}$ is $8$-dimensional. This means that for any covariantly constant endomorphisms $A,B,C,D$ of $T \mathcal{SU}_{n,d}$, there is a covariantly constant endomorphism $\phi \in \mathcal{C}$ satisfying Equation (\ref{equmatrix1}).\\

Consider the case $A = Id$, $B = C = 0$, $D = -Id$. The corresponding $\phi \in \mathcal{C}$ is trace-free and satisfies $\phi^2 = Id$. Let $\mathcal{T}_+,\mathcal{T}_-$ be the $\pm 1$-eigenbundles of $\phi$. Thus $T \mathcal{M}_{n,d} = \mathcal{T}_+ \oplus \mathcal{T}_-$ and $\mathcal{T}_+ |_{\mathcal{SU}_{n,d}} = T\mathcal{SU}_{n,d}$, $\mathcal{T}_- |_{\mathcal{SU}_{n,d}} = N\mathcal{SU}_{n,d}$. This implies that locally $\mathcal{M}_{n,d}$ is isometric to a product. Moreover, around any point $m \in \mathcal{SU}_{n,d}$ we also have that $\mathcal{M}_{n,d}$ is isometric to a product $U_m \times V_m$ of a neighborhood $U_m$ of $m$ in $\mathcal{SU}_{n,d}$ with another space $V_m$. Since $\mathcal{M}_{n,d}$ is hyper-K\"ahler, it is Ricci flat and hence this implies $U_m,V_m$ are both Ricci flat. As the point $m$ was arbitrary, this would imply that $\mathcal{SU}_{n,d}$ with its natural metric is Ricci flat, but this is impossible, since the anti-canonical bundle is ample \cite{ram}.
\end{proof}
%%%%%%%%%%%%%%%%%%%%%%%%%%%%%%%%%%%%%%%

%%%%%%%%%%%%%%%%%%%%%%%%%%%%%%%%%
% COROLLARY %%%%
%%%%%%%%%%%%%%%%%%%%%%%%%%%%%%%
\begin{corollary}
Suppose $n$ and $d$ are coprime. The Riemannian holonomy group of $\mathcal{M}_{n,d}$ is $Sp(m)$, $m = (n^2-1)(g-1)$.
\end{corollary}
\begin{proof}
Let $G \subseteq Sp(m)$ be the Riemannian holonomy group of $\mathcal{M}_{n,d}$ and $G^0 \subseteq G$ the identity component. Then $G^0$ is a closed Lie subgroup of $Sp(m)$ \cite{boli}. If $G^0$ is a proper subgroup then it must act reducibly on $T \mathcal{M}_{n,d}$. But this would contradict Proposition \ref{propconstantendo}, so in fact $G^0 = G = Sp(m)$.
\end{proof}
%%%%%%%%%%%%%%%%%%%%%%%%%%%%%%%%%

%%%%%%%%%%%%%%%%%%%%%%%%%%%%%%%%%%%%%%
% PROPOSITION
%%%%%%%%%%%%%%%%%%%%%%%%%%%%%%%%%%
\begin{theorem}
Suppose $n$ and $d$ are coprime.
\begin{itemize}
\item{If $\Sigma$ does not admit an anti-automorphism, then every isometry of $\mathcal{M}_{n,d}$ preserves $I$. Therefore $Isom(\mathcal{M}_{n,d}) = Aut_{Isom}(\mathcal{M}_{n,d}) \cong \left( U(1) \times Aut( \mathcal{SU}_{n,d} ) \right)$.}
\item{If $\Sigma$ admits an anti-automorphism then the subgroup of isometries of $\mathcal{M}_{n,d}$ preserving $I$ has index $2$ in the isometry group of $\mathcal{M}_{n,d}$.}
\end{itemize}
\end{theorem}
\begin{proof}
Let $\phi : \mathcal{M}_{n,d} \to \mathcal{M}_{n,d}$ be an isometry. Then $\phi^*(I)$ is a covariantly constant complex structure, hence by Proposition \ref{propconstantendo}, $\phi^*(I)$ belongs to the $2$-sphere of complex structures $\{ aI + bJ + cK \, | a^2 + b^2 + c^2 = 1 \, \}$. However, it is known that $I$ is not isomorphic to any complex structure in this $2$-sphere other than itself and $-I$ \cite{hit1}. It follows that either $\phi^*(I) = I$ or $-I$. If there exists an isometry $\phi$ for which $\phi^*(I) = -I$ then by Theorem \ref{theoremantiauto}, $\Sigma$ admits an anti-automorphism. Conversely if $\Sigma$ admits an anti-automorphism $f : \Sigma \to \Sigma$ then, as in the proof of Theorem \ref{theoremantiauto} we constructed an anti-automorphism $\hat{f}_L$ of $\mathcal{M}_{n,d}$. It is easy to see that this is an isometry.
\end{proof}
%%%%%%%%%%%%%%%%%%%%%%%%%%%%%%%%%%%%

%%%%%%%%%%%%%%%%%%%
% REMARK %%%%%%%%%%%%
%%%%%%%%%%%%%%%%%%%
\begin{remark}\label{remfullisom}
In the case that $\Sigma$ admits an anti-automorphism $f : \Sigma \to \Sigma$ we can be more precise about the isometry group of $\mathcal{M}_{n,d}$. Recall as in the proof of Theorem \ref{theoremantiauto}, we constructed an anti-automorphism $\hat{f}_L : \mathcal{M}_{n,d} \to \mathcal{M}_{n,d}$ which is also an isometry. Then
\begin{equation*}
Isom(\mathcal{M}_{n,d}) = \left( U(1) \times Aut( \mathcal{SU}_{n,L_d} ) \right)  \cup \hat{f}_L \circ \left( U(1) \times Aut( \mathcal{SU}_{n,L_d} ) \right).
\end{equation*}
To completely describe the group structure of $Isom(\mathcal{M}_{n,d})$ one just needs to know (1) the element ${\hat{f}_L}^2 \in \left( U(1) \times Aut( \mathcal{SU}_{n,d} ) \right)$ and (2) the adjoint action of $\hat{f}_L$ on the subgroup $U(1) \times Aut( \mathcal{SU}_{n,d} )$. For (1) let $\sigma : \Sigma \to \Sigma$ be the automorphism $\sigma = f^2$ and set $M = L \otimes f^*(L^*)$. Then it is easy to see that $M^n \cong L_0 \otimes \sigma^*(L_0^*)$ and ${\hat{f}^2}_L$ is the automorphism of $\mathcal{M}_{n,d}$ given by $(E,\Phi) \mapsto (\sigma^*(E) \otimes M , \sigma^*(\Phi) \otimes Id )$. For (2) we find that the adjoint action of $\hat{f}_L$ on the factor $U(1)$ is complex conjugation. The adjoint action of $\hat{f}_L$ on $Aut(\mathcal{SU}_{n,d})$ can be easily determined from the description of $Aut(\mathcal{SU}_{n,d})$ given in Theorem \ref{theoremstableauto}.
\end{remark}
%%%%%%%%%%%%%%%%%%%%

%%%%%%%%%%%%%%%%%%%%%%%%%%%%%%%%%%%%%%%%%%%%%%%%%%%%%%%%%%%%%%%%%%%%%%%%%%%%%%%%
%%%%%%%%%%%%%%%%%%%%%%%%%%%%%%%%%%%%%%%%%%%%%%%%%%%%%%%%%%%%%%%%%%%%%%%%%%%%%%%%
%%%%%%%%%%%%%%%%%%%%%%%%%%%%%%%%%%%%%%%%%%%%%%%%%%%%%%%%%%%%%%%%%%%%%%%%%%%%%%%%
%%%%%%%%%%%%%%%%%%%%%%%%%%%%%%%%%%%%%%%%%%%%%%%%%%%%%%%%%%%%%%%%%%%%%%%%%%%%%%%%

\bibliographystyle{amsplain}

\end{document}